\documentclass[BCOR8mm,DIV14]{scrartcl}
\usepackage{amsmath, amsthm, amscd, amsfonts, amssymb, graphicx, color}
\newtheorem{theorem}{Theorem}[section]
\newtheorem{lemma}[theorem]{Lemma}
\newtheorem{proposition}[theorem]{Proposition}
\newtheorem{corollary}[theorem]{Corollary}
\newtheorem{definition}[theorem]{Definition}

\newtheorem{conjecture}[theorem]{Conjecture}

\newtheorem{remark}[theorem]{Remark}
\numberwithin{equation}{section}

\begin{document}
\setcounter{page}{1}

%      The  following line is only  for the journal  'Annals of Functional Analysis'
%-------------------------- Pleased do not change the following line-------------------------------------------
%\noindent \textcolor[rgb]{0.99,0.00,0.00}{This is a submission to one of journals of TMRG: BJMA/AFA}\\[.5in]
%--------------------------------------------------------------------------------------------------------------

\title{The Complex Angle in Normed Spaces}
%\title[Short Title]{Angles and a Classification of Normed Spaces}

%\author[VOLKER W. TH\"UREY]{VOLKER W. TH\"UREY$^{1}$}  %%% Author for 'Annals of Functional Analysis'  
\author{VOLKER W. TH\"UREY  \\  Rheinstr. 91  \\   28199 Bremen,  Germany    
     \thanks{T: 49 (0)421/591777 \  \   E-Mail: volker@thuerey.de }   } 
 \maketitle
%  \address{$^{1}$ Rheinstr. 91, 28199 Bremen, Germany. \ T: 49(0)421/591777} 
%                                                   %%% Address for 'Annals of Functional Analysis' 
%\email{\textcolor[rgb]{0.00,0.00,0.84}{volker@thuerey.de;first2@afa.ac.ir}
%\email{\textcolor[rgb]{0.00,0.00,0.84}{volker@thuerey.de}}  %%% E-Mail for 'Annals of Functional Analysis' 
%\newline
 
%\dedicatory{This paper is dedicated to Professor ABCD}  

%\subjclass[2010]{Primary 46B20; Secondary 52A10.}   %%% Subjectclass for 'Annals of Functional Analysis' 

%\keywords{generalized angle, normed space, inner product space}  %%% Keywords for 'Annals of Functional Analysis' 

%\date{Received: xxxxxx; Revised: yyyyyy; Accepted: zzzzzz.  }
%\newline     \indent $^{*}$ Corresponding author}

 \begin{abstract}
  We consider a generalized angle in complex normed vector spaces.
        Its definition corresponds to the definition of the well known Euclidean angle in  real inner product spaces.
        Not surprisingly it yields complex values as `angles'. This `angle' has
        some simple  properties, which are known from the usual angle in real inner product spaces. 
        But to do ordinary Euclidean geometry real angles are necessary. 
        We show that even in a complex normed space     %   many of the angles have pure real values. 
        there are many pure real valued `angles'. The situation improves yet in inner product spaces.
        There we can use the known theory of orthogonal systems to find many pairs of vectors with real angles, 
        and to do geometry which is  based on the Greeks 2000 years ago. %%  \textperthousand
        %%%%%%%%%%%%%%%%%%%%%%%%%%%%%%%%%%%%%%%%%%%%%%%%%%%%%%%%%%%%%%%%%%%%%%%%%%%%%%%%%%%%%%%%%%%%
  %   {\it Keywords and phrases}: complex normed space, generalized angle, complex Hilbert space  \\ 
  %   {\it Mathematical Subject Classification}: 46B20, 46C05, 30E99      
 \end{abstract}
 %%%%%%%%%%%%%%%%%%%%%%%%%%%%%%%%%%%%%%%%%%%%%%%%%%%%%%%%%%%%%%%%%%%%%%%%%%%%%%%%%%%%%%%%%%%%%%%%%%%%  
  \tableofcontents     
 %%%%%%%%%%%%%%%%%%%%%%%%%%%%%%%%%%%%%%%%%%%%%%%%%%%%%%%%%%%%%%%%%%%%%%%%%%%%%%%%%%%%%%%%%%%%%%%%%%%%   
  	    \section{Introduction}      \ \ \  \ \ \ \   
  	 We deal with complex vector spaces   $ X $ provided with a norm  $ \| \cdot \|$. To initiate the following
  	 constructions we begin with the special case of an inner product space \                                
	   $ \left( X , <  . \, | \, . > \right) $ over the complex field ${\mathbb C}$.  It is well-known that the inner
	   product can be expressed by the norm, namely for $ \vec{x} , \vec{y} \in X $ we can write 
	 \begin{align}    \label{allererste definition}
	       <\vec{x} \ | \ \vec{y}> \ = \  
         \frac{1}{4} \cdot \left[ \: \|\vec{x}+\vec{y}\|^{2} -   \|\vec{x}-\vec{y}\|^{2} \: 
         + {\bf i} \cdot \left( \: \|\vec{x} + {\bf i} \cdot \vec{y}\|^{2} -  
         \|\vec{x} - {\bf i} \cdot \vec{y}\|^{2} \: \right) \: \right] \, , 
   \end{align}      
         where the symbol `{\bf i}' means the imaginary unit. \\
         For two vectors $ \vec{x}, \vec{y} \neq \vec{0} $  it holds \ \
         $ <\vec{x} \ | \ \vec{y}> \ = \ \|\vec{x}\| \cdot \|\vec{y}\| \cdot 
          < \frac{ \vec{x} }{\|\vec{x}\|} \ | \ \frac{\vec{y}}{\|\vec{x}\|}> $.   \ 
    We use both facts and an idea in \cite{Singer} to generate a continuous product in all complex normed
    vector spaces   $ (X, \| \cdot \|)$, which is just the inner product in the special case of a complex 
    inner product space.      
          
  %  This means for two vectors $ \vec{x}, \vec{y} \neq \vec{0} $ that we have the expression  \ \
  %        $ <\vec{x} \ | \ \vec{y}> \  =  $          
	%  \begin{equation*}  
  %        \frac{1}{4} \ \cdot \ \|\vec{x}\| \ \ \cdot \|\vec{y}\| \ \cdot \
  %        \left[  \ \left\| \frac{\vec{x}}{\|\vec{x}\|}   +  \frac{\vec{y}}{\|\vec{y}\|} \right\|^{2}  - 
  %        \left\| \frac{\vec{x}}{\|\vec{x}\|}   -   \frac{\vec{y}}{\|\vec{y}\|} \right\|^{2}   
  %         + {\bf i} \cdot \left( \ \left\| \frac{\vec{x}}{\|\vec{x}\|} + 
  %          {\bf i} \cdot \frac{ \vec{y}}{\|\vec{y}\|} \right\|^{2} - \left\| \frac{\vec{x}}{\|\vec{x}\|} - 
  %          {\bf i} \cdot \frac{ \vec{y}}{\|\vec{y}\|} \right\|^{2} \ \right) \right] \ .  
  %   \end{equation*} 
  %  
  %                              
  %  In this  paper we deal with complex normed vector spaces $ ( X, \| \cdot \| )$.  Following the lines of 
  %  an inner product we   define a continuous product \  $ < . \: | \: . > $ \ on $ X $. 
   \begin{definition}   \label{die erste definition}         
               Let   $ \vec{x}, \vec{y}$ be two arbitrary elements 
               of $ X $. In the case of  \ $ \vec{x} = \vec{0} $  or  $ \vec{y} = \vec{0}  $  we set  
               \ $ < \vec{x} \: | \:  \vec{y} > \ := \, 0 $,  and if  \
               $ \vec{x} , \vec{y} \neq \vec{0} \ ( \text{i.e.} \ \|\vec{x}\| \cdot \|\vec{y}\| > 0 \; )$ 
      \ we define the complex number   \\  \\
               \centerline { $ < \vec{x} \: | \:  \vec{y} >  \ :=  $  } 
        $$        \|\vec{x}\| \cdot \|\vec{y}\| \cdot  \frac{1}{4} \cdot  
                  \left[ \left\| \frac{\vec{x}}{\|\vec{x}\|}   +  \frac{\vec{y}}{\|\vec{y}\|} \right\|^{2}   - 
                  \left\| \frac{\vec{x}}{\|\vec{x}\|} - \frac{\vec{y}}{\|\vec{y}\|} \right\|^{2} +
                  {\bf i} \cdot  \left( \left\| \frac{\vec{x}}{\|\vec{x}\|}  +  
                  {\bf i} \cdot \frac{ \vec{y}}{\|\vec{y}\|} \right\|^{2} - \left\| \frac{\vec{x}}{\|\vec{x}\|} -
                  {\bf i} \cdot \frac{ \vec{y}}{\|\vec{y}\|} \right\|^{2}  \right)  \right] \ . 
                                                                $$    $  {  }  \hfill  \Box  $  
   \end{definition}        
        It is easy to show that the product fulfils  the conjugate symmetry  ($ <\vec{x} | \vec{y}>$ 
        $ \ = \ \overline{<\vec{y} | \vec{x}>} $), where $\overline{<\vec{y} | \vec{x}>}$ means the complex
        conjugate, the positive definiteness ($ < \vec{x} | \vec{x} > \ \geq \ 0 \
        \text{and also}$  $ < \vec{x} | \vec{x} > = 0 \ \text{only for} \ \vec{x} = \vec{0} $), and  the homogeneity 
        for real numbers, ($  < r \cdot \vec{x} \, | \, \vec{y} > \ = \ r \cdot <  \vec{x} \, | \, \vec{y} > $), 
        and  the homogeneity for pure imaginary numbers, \ ($  < r \cdot {\bf i} \cdot \vec{x} \, | \, \vec{y} > \ 
        = \ r \cdot {\bf i} \cdot <  \vec{x} \, | \, \vec{y} > $), \ for 
        $ \vec{x},  \vec{y} \in X , \ r \in {\mathbb R}$. \ Further, for $ \vec{x} \in X $ it holds \
        $ \|\vec{x}\| = \sqrt{< \vec{x} \: | \: \vec{y} >}$.      
                                                
       If we switch for a moment to real inner product spaces  \ $ \left( X , <  . \, | \, . >_{real} \right) $ \
       we have  for all  $ \vec{x},\vec{y} \neq \vec{0} $ \ the  usual Euclidean angle  \
   %%       $\angle_{Euclid} (\vec{x}, \vec{y}) \ = $
   \begin{align*}       \angle_{Euclid} (\vec{x}, \vec{y}) \ = \
         \arccos{ \left( \frac{< \vec{x} \: | \: \vec{y} >_{real} }{ \|\vec{x}\| \cdot \|\vec{y}\| } \right) }   
         \  = \  \arccos{ \left( \ \frac{1}{4} \cdot   
         \left[ \ \left\| \frac{\vec{x}}{\|\vec{x}\|} + \frac{\vec{y}}{\|\vec{y}\|} \right\|^{2} \ - \
         \left\| \frac{\vec{x}}{\|\vec{x}\|} -  \frac{\vec{y}}{\|\vec{y}\|} \right\|^{2} \ \right] \ \right) } \ , 
   \end{align*}    
        which can be defined in terms of the norm, too.   
                                                                
         For two vectors  $ \vec{x}, \vec{y} \neq \vec{0}$  from a complex normed vector space $ ( X, \| \cdot \| )$ 
	       we are able to define an `angle'  which  coincides with the definition of the Euclidean angle in 
	       real inner product spaces. 
	 \begin{definition}   \label{die zweite definition}   
	      Let  $ \vec{x}, \vec{y}$ be two  elements  of $ X \backslash \{ \vec{0}\}$. \ We define the complex  number 
	    \begin{equation*}   
	        \angle (\vec{x}, \vec{y}) \ := \
          \arccos{ \left( \frac{< \vec{x} \: | \: \vec{y} > } { \|\vec{x}\| \cdot \|\vec{y}\| }  \right) } \ .
      \end{equation*}        
       This number $ \angle (\vec{x}, \vec{y}) \in {\mathbb C} $   is called the {\it angle} 
	     of the pair $ (\vec{x}, \vec{y}) $.   \hfill     $\Box$ 
   \end{definition}  
      We state that  the angle 
        $ \angle  (\vec{x}, \vec{y}) $ is defined for all \  $ \vec{x}, \vec{y} \neq \vec{0} $. 
  %    We state that in the case that  $ < \vec{x} \: | \: \vec{y} > $ is real, the                          
  % 	   {\it Cauchy-Schwarz-Bunjakowsky Inequality} or {\sf CSB }inequality is fulfilled,  that means  
  %      we have the inequality   
  %  $$ |< \vec{x} \: | \: \vec{y} >   |  \ \leq \ \|\vec{x}\| \cdot \|\vec{y}\|  \ .$$ 
  %      Since the real interval $ [ -1, 1 ] $ is a subset of ${\cal B}$,  we get that the angle 
  %      $ \angle  (\vec{x}, \vec{y}) $ is defined for all \  $ \vec{x}, \vec{y} \neq \vec{0} $. 
         Since we deal with complex vector spaces it is not surprising that we get complex numbers as `angles'. 
	       For the definition we need the extension of the cosine and  arccosine functions on complex numbers.  
	       We use two subsets $ {\cal A} $ and $ {\cal B} $ of the complex plane ${\mathbb C}$, where 
	   \begin{align}
	         & {\cal A} \ := \  \{ a + {\bf i} \cdot b \; \in {\mathbb C}  \ | \ 0 < a < \pi , \ b \in {\mathbb R} \} 
	         \ \cup \ \{ 0 , \pi \} , \ \ \text{and}   \\ 
	         & {\cal B} \ := \ {\mathbb C} \backslash \{ r \in {\mathbb R} \: | \: r < -1 \ \ \text{or} \ \ r > +1 \}. 
     \end{align}   
	        We have two known homeomorphisms 
	   \begin{align*}
	         & \cos: \ {\cal A} \stackrel{\cong}{\longrightarrow} {\cal B} \ \ \ \text{and} \ \  
	          \arccos: \ {\cal B} \stackrel{\cong}{\longrightarrow} {\cal A} \ .   \ \
	    %      \text{and we get two identities}  \\  
	    %      & \cos \circ \arccos = id_{{\cal B}} \ \ \ \text{and} \ \ \arccos \circ \cos = id_{{\cal A}} \ .
	   \end{align*}
	       The cosines of the `angles' are  are in the `complex square' \
         $ {\cal SQ} :=  \{ r + {\bf i} \cdot s \in {\mathbb C} \ | \ -1 \leq r, s \leq +1 \}  \subset {\cal B} $.
         The values of the `angles' are from its image $ \arccos({\cal SQ}) $, which forms a symmetric hexagon in
         $ {\cal A} $.  Its center is $ \pi/2$, two corners are $ 0 $ and $ \pi $.  A third is, for instance, \\
      \centerline{ $ \arccos(1 + {\bf i}) \ = \ (\pi/2) - (1/2) \cdot \left[ \arccos\left(\sqrt{5}-2\right) +
         {\bf i} \cdot \log \left(  \sqrt{5}+2 + 2 \cdot \sqrt{\sqrt{5}+2}  \right)\right] \ 
         \approx \ 0.90 - {\bf i} \cdot 1.1 $. }   
   %        The values of the `angles' are out of the set $ {\cal A} $, while the set of their cosines are
   %        a subset of the `complex square'
   %        $\{ r + {\bf i} \cdot s \in {\mathbb C} \ | \ -1 \leq r, s \leq +1 \} \subset {\cal B} $. \
                                                    
	         This `angle' has eight comfortable properties ({\tt An} 1) - ({\tt An} 8) which are 
	         known from the   Euclidean angle in real inner product spaces.
  %%%%%%%%%%%%%%%%%%%%%%%%%%%%%%%%%%%%%%%%%%%%%%%%%%%%%%%%%%%%%%%%%%%%%%%%%%%%%%%%%%%%%%%%%%%%%
   % \newpage 
  \begin{theorem}  \label{allererstes theorem}
    	In a complex normed space   $ (X, \| \cdot \|) $ the angle $ \angle $ has the	following properties:         
    \begin{itemize}
          \item ({\tt An} 1)   \quad $\angle  \ $  is a continuous map from   
                $ \left( X \backslash \{\vec{0}\} \right)^{2} $ onto a subset of \
                $ \arccos({\cal SQ}) \subset {\cal A} $. \\ 
             For elements   $ \vec{x},\vec{y} \neq \vec{0}$   it holds that                     
          \item ({\tt An} 2)   \quad  $\angle                 (\vec{x}, \vec{x}) = 0 $,  
          \item ({\tt An} 3)   \quad  $\angle                  ( -\vec{x}, \vec{x}) = \pi$,   
          \item ({\tt An} 4)   \quad  $\angle                  (\vec{x}, \vec{y}) 
                               =              \overline{\angle (\vec{y}, \vec{x})} $, 
          \item ({\tt An} 5)   \quad  for real numbers \ $ r,s > 0 $   we  have    
                               $\angle                 (r \cdot \vec{x}, s \cdot \vec{y}) 
                               = \angle                  (\vec{x},\vec{y})$, 
          \item ({\tt An} 6)   \quad  $\angle                 ( - \vec{x}, - \vec{y}) 
                               = \angle                  (\vec{x},\vec{y})$,    
          \item ({\tt An} 7)   \quad  $ \angle                 (\vec{x}, \vec{y}) + 
                               \angle                 (-\vec{x}, \vec{y})  = \pi $.  
          \item ({\tt An} 8)   \quad
          \quad   For any two  linear independent vectors \ $ \vec{x},\vec{y}$ \ of  \  
           $( X , \| \cdot \| )$ \  there is  a continuous injective map \ \     
                         $  \Theta: \ \mathbb{R} \longrightarrow  {\cal A}  , \ \
                           t \mapsto \angle (\vec{x}, \vec{y} + t \cdot\vec{x})$.     \ \
   %    For each pair of linear independent elements $ \vec{x},\vec{y} \in X$ we have the following 
   %    homeomorphism $ \lambda $ with domain $\mathbb{R}$ onto the open interval  $\left( 0, \pi \right)$, \\ 
   %     \centerline{  $ \lambda: \mathbb{R} \stackrel{\cong}{\longrightarrow} \left( 0, \pi \right)  , \ \
   %                    t \mapsto \ \text{the real part of} \ \angle (\vec{x}, \vec{y} + t \cdot\vec{x})$ . }  \\  
   %    The map \ $ \lambda $ \ is strictly decreasing. 
       The limits are \\ \centerline{ $ \lim_{t\rightarrow -\infty} \ \Theta (t) = \pi $ \ and \  
                                 $ \lim_{t\rightarrow \infty} \ \Theta (t) = 0$. }                       
  \end{itemize} 
  \end{theorem}   {  $ $ }  % \\ 
    In the following table we note some angles and their cosines.  
      We choose two elements $ \vec{x},\vec{y} \neq \vec{0} $ \ of a complex normed space $ ( X, \| \cdot \| )$, 
      and six suitable real numbers \ $ a, b, r, s, v, w $ with 
      $ - \frac{\pi}{2} \leq a, v \leq \frac{\pi}{2}$ \ and \ $ -1 \leq r, s \leq 1 $, such that  \\ 
       \centerline {   $ \angle (\vec{x},\vec{y}) =: \frac{\pi}{2} + a + {\bf i} \cdot b \in {\cal A}$,   } \\  
       \centerline {  $ \cos(\angle (\vec{x},\vec{y})) = 
       \cos \left(\frac{\pi}{2} + a + {\bf i} \cdot b \right) =: r + {\bf i} \cdot s \in {\cal B}$, \ \ and    }  \\  
       \centerline {  $ \angle ({\bf i} \cdot \vec{x},\vec{y}) =: 
                               \frac{\pi}{2} + v + {\bf i} \cdot w \in {\cal A}$. } \\ 
      Note that the cosines of all angles in the table (third column)  have the same modulus \
      $\sqrt{ r^{2} + s^{2}} $.  \\        
   \begin{center}  
   \begin{tabular}{c|l|l|c|c}     \hline
        pair of vectors  & their angle $ \angle $ & the cosine of $ \angle $ & the angle for $ \vec{x} = \vec{y} $ 
                                                                         & its cosine for $ \vec{x} = \vec{y} $   \\ 
        \hline\hline
        $(\vec{x},\vec{y})$  & $ \frac{\pi}{2} + a + {\bf i} \cdot b $     
                             &  $ r + {\bf i} \cdot s $ & $ 0 $ & $ 1 $  \\
        $(-\vec{x},\vec{y})$  & $ \frac{\pi}{2} - a - {\bf i} \cdot b $     
                              & $ -r - {\bf i} \cdot s $ & $ \pi $ & $ -1 $   \\ 
        $(\vec{y},\vec{x})$  & $  \frac{\pi}{2} + a - {\bf i} \cdot b $  
                           & $ r - {\bf i} \cdot s $ & $ 0 $ & $ 1 $     \\
        $(-\vec{y},\vec{x})$  & $  \frac{\pi}{2} - a + {\bf i} \cdot b $    
                              & $ -r + {\bf i} \cdot s $   & $ \pi $ & $ -1 $   \\   \hline 
        $({\bf i} \cdot \vec{x},\vec{y})$  & $  
                   \frac{\pi}{2} + v + {\bf i} \cdot w  $ &  $ -s + {\bf i} \cdot r $ & 
                    $\frac{\pi}{2} - {\bf i} \cdot \log \left[ \sqrt{2}+1 \right]$  &  ${\bf i}$  \\
        $(\vec{y},{\bf i} \cdot \vec{x})$  & $ 
                   \frac{\pi}{2} + v - {\bf i} \cdot w  $ &  $ -s - {\bf i} \cdot r $ &
                    $\frac{\pi}{2} + {\bf i} \cdot \log \left[ \sqrt{2}+1 \right]$  &  $-{\bf i}$   \\ 
        $(\vec{x},{\bf i} \cdot \vec{y})$  & $    
                   \frac{\pi}{2} - v - {\bf i} \cdot w $ &  $  s - {\bf i} \cdot r $ &
                      $\frac{\pi}{2} + {\bf i} \cdot \log \left[ \sqrt{2}+1 \right]$  &  $-{\bf i}$   \\   
        $({\bf i} \cdot \vec{y},\vec{x})$  & $   
                   \frac{\pi}{2} - v + {\bf i} \cdot w $ &  $  s + {\bf i} \cdot r $ &
                      $\frac{\pi}{2} - {\bf i} \cdot \log \left[ \sqrt{2}+1 \right]$  &  ${\bf i}$   \\  \hline
   %     $( -{\bf i} \cdot \vec{y},\vec{x})$  & $    
   %                \frac{\pi}{2} + v - {\bf i} \cdot w $ &  $ -s - {\bf i} \cdot r $ &
   %                   $\frac{\pi}{2} + {\bf i} \cdot \log \left[ \sqrt{2}+1 \right]$  &  $-{\bf i}$   \\   \hline 
   \end{tabular} 
   \end{center}     $ { } $    \\ \\      
        With given two vectors $ \vec{x},\vec{y} \in X $ we consider as before the angle 
        $  \angle ( \vec{x},\vec{y}) = \frac{\pi}{2} + a + {\bf i} \cdot b $ with suitable real numbers $ a, b $.
        Now we express the complex number $ \angle ({\bf i} \cdot \vec{x},\vec{y}) $ in dependence of $ a $ and $ b $.      %    We assume two vectors $ \vec{x},\vec{y} $ and an angle 
    %    $  \angle ( \vec{x},\vec{y}) = \frac{\pi}{2} + a + {\bf i} \cdot b $ with suitable real numbers $ a, b $.  
    %    We want to express the complex number \ $  \angle ({\bf i} \cdot \vec{x},\vec{y}) $ \
    %    in coordinates of $ a $ and $ b $. \
        For the presentation the real valued cosine and hyperbolic cosine are used, and their inverses
        arccosine and area hyperbolic cosine. We use abbreviations like $ \cos, \cosh, \arccos $, and arcosh.  \  
        The result is as follows:   
   \begin{theorem}   \label{theorem winkelixy}
    
       In a complex normed space  $ ( X, \| \cdot \| )$ we take two elements 
         $ \vec{x}, \vec{y} \neq \vec{0}$.  We assume the angle  \
         $ \angle (\vec{x},\vec{y}) = \frac{\pi}{2} + a + {\bf i} \cdot b \in {\cal A} \, , \
        \text{i.e.} - \frac{\pi}{2} \leq a \leq  \frac{\pi}{2}$. \ $($If \ $ a = - \frac{\pi}{2}$ \ or \ 
        $ a = \frac{\pi}{2}$ \ since \ $ \angle (\vec{x},\vec{y}) \in {\cal A} $  \ it follows \ $ b = 0 $.$)$ 
        \  We get
        \begin{align}
          \angle ({\bf i} \cdot \vec{x},\vec{y}) \ = \ %   \frac{\pi}{2} + v + {\bf i} \cdot w  
            \frac{\pi}{2} + \frac{1}{2} \cdot \left[ - {\rm sgn}(b) \cdot \arccos({\mathsf{H}}_-) +
            {\bf i} \cdot {\rm sgn}(a) \cdot {\rm arcosh}({\mathsf{H}}_+ ) \right] , 
        \end{align}    
           with the abbreviations \ $ {\mathsf{H}}_- $ \ and \ $ {\mathsf{H}}_+ $, where 
        \begin{align*}   
          & {\mathsf{H}}_{\pm} := \sqrt{\left[  \cos^{2}\left(\frac{\pi}{2} + a\right) + \cosh^{2}(b) - 2 \right]^{2}
            + 4 \cdot  \cos^{2}\left(\frac{\pi}{2} + a\right) \cdot \cosh^{2}(b) }  \\
           & \qquad \qquad \qquad  \qquad \qquad \qquad  \qquad  \qquad  \qquad \qquad \qquad  \qquad \quad
            \pm \left[  \cos^{2}\left(\frac{\pi}{2} + a\right) + \cosh^{2}(b) - 1  \right] \ . 
        \end{align*}
   \end{theorem} {  $ $ }   \\
     It is worthwhile to look at special cases. We consider a real angle \
        $ \angle (\vec{x},\vec{y}) = \frac{\pi}{2} + a $ (i.e. $ b = 0 $), and an angle on the vertical axis 
        $ x = \frac{\pi}{2}$, \ this means \ $ \angle (\vec{x},\vec{y}) = \frac{\pi}{2} + {\bf i} \cdot b $
        (i.e. $ a = 0 $). 
   \begin{corollary}  \label{corollary pure real angle}
    For a pure real angle \ $ \angle (\vec{x},\vec{y}) = \frac{\pi}{2} + a $ \ with \
     $ -\frac{\pi}{2} \leq a \leq  \frac{\pi}{2} $, i.e. $ b = 0 $, we get a complex angle \
     $\angle({\bf i} \cdot \vec{x},\vec{y})$  with a real part \ $ \pi/2 $,  
   \begin{align*}  
         \angle({\bf i} \cdot \vec{x},\vec{y}) % \ = \ \frac{\pi}{2} + v + {\bf i} \cdot w 
         \ & = \  \frac{\pi}{2} +  {\bf i} \cdot \frac{1}{2} \cdot {\rm sgn}( a ) \cdot 
               {\rm arcosh} \left( 2 \cdot \cos^{2} \left(\frac{\pi}{2} + a \right) + 1 \right)   \\
         \ & = \   \frac{\pi}{2} + {\bf i} \cdot {\rm sgn}( a ) \cdot
         \log \left[ \ \sqrt{ \cos^{2}\left(\frac{\pi}{2} + a \right) + 1}  
         + \left|\cos \left(\frac{\pi}{2} + a \right)\right| \ \right]  \ .  
   \end{align*}       
   \end{corollary}  
   \begin{corollary} \label{corollary real part pidurchzwei angle}
        For an  angle $ \angle (\vec{x},\vec{y}) = \frac{\pi}{2} + {\bf i} \cdot b $ 
        $( \text{i.e.} \ a = 0 )$ \ we get a pure real angle $\angle({\bf i} \cdot \vec{x},\vec{y})$,
      \begin{align*}   
         \angle({\bf i} \cdot \vec{x},\vec{y}) \  
         & = \ \ \frac{\pi}{2} - \frac{1}{2} \cdot {\rm sgn}( b ) \cdot \arccos[ 3 - 2 \cdot \cosh^{2}(b)]  \\
         & = \ \ \frac{\pi}{2} - \frac{1}{2} \cdot {\rm sgn}( b ) \cdot \arccos[ 2 - \cosh(2 \cdot b)]  
         \ = \  \frac{\pi}{2} - \arcsin[ \sinh(b) ]  \  = \  \arccos[ \sinh(b) ]   \ .
      \end{align*}   
   \end{corollary}   
  %  \\  \\
    Now we look for pairs  $ (\vec{x}, \vec{y}) $ with a real valued angle $ \angle(\vec{x}, \vec{y}) $.   \    
        For a complex normed vector space $ (X, \| \cdot \|) $ we define the set of pairs \
    $$ {\cal R}_X^{\ {\bf \bullet}} := \{(\vec{x}, \vec{y}) \ | \  \vec{x}, \vec{y} \in X, \vec{x}, \vec{y} \neq
          \vec{0},  \ \text{and} \ \angle(\vec{x}, \vec{y})  \in {\mathbb R} \}  \  .     $$ 
     Let us take two vectors $ \vec{x}, \vec{y} \neq  \vec{0} $. We can prove that there is a real number 
     $ \varphi \in [ 0, 2 \pi ] $ such that \
     $ ( e^{{\bf i} \cdot \varphi} \cdot \vec{x} \,, \vec{y}) \in  {\cal R}_X^{\ {\bf \bullet}} $, i.e. the pair has 
     a pure real angle \ $ \angle( e^{{\bf i} \cdot \varphi} \cdot \vec{x} \, , \, \vec{y} )$. 
                                               
     This fact ensures the existance of  many real valued angles even in complex normed spaces. The situation 
     will improve further in the special case of complex vector spaces provided with an inner product
     $ < . \: | \: . > $.  We discuss this case in the next section.  \\
                                                                                
     The properties of complex inner product spaces  $( X , < . \: | \: . > ) $ have been studied extensively, and 
     such spaces have many applications in technology and physics. 
                                                                                
     To do ordinary Euclidean geometry we need real valued angles.
     The idea is simple. We take an orthogonal basis \ $ \mathsf{T} $ of  $( X , < . \: | \: . > ) $,
     and we generate $  {\cal L}(\mathbb R)(\mathsf{ T }) $, the set of all finite real linear combinations
     of elements of $ \mathsf{T} $.  Let 
   $$     {\cal L}(\mathbb R)(\mathsf{ T }) \  := \ 
       \left\{ \sum_{i=1}^{n} r_i \cdot \vec{x}_i \ | \ n \in \mathbb N, \ r_1, r_2, \ldots , r_n \in \mathbb R, \
       \vec{x}_1, \vec{x}_2 , \ldots , \vec{x}_n \in  \mathsf{ T } \right\} \ .     $$
     If $ X $ is a Hilbert space, i.e. it is complete, we even can use Cauchy sequences $($with real coefficients$)$ 
     from elements of $ \mathsf{T} $. It means that we can take the closure \
     $ \overline{{\cal L}(\mathbb R)(\mathsf{T})} $ \ of \ ${\cal L}(\mathbb R)(\mathsf{T})$.   This creates a 
     real linear subspace \ $ \overline{{\cal L}(\mathbb R)(\mathsf{T})} $ \ of $X$, in which   % where
     all angles are real.         
                                                                                                    
     After that we consider complex vector spaces $ X $ of finite complex dimension $n \in {\mathbb N}$. 
     Their real dimension is \ $ 2 \cdot n $, and we state in Proposition \eqref{endlich dimensionale vektorraeume}
     that the maximal dimension  of a real subspace of $ X $ with all-real angles is $ n $. 
     The real span $ {\cal L}(\mathbb R)(\mathsf{ T })$ of an  orthogonal basis $ \mathsf{ T } \subset X$ 
     yields an example.   
                                                                                         
     Finally we demonstrate two examples of ordinary Euclidean geometry in complex inner product spaces,
     and to do this we show that real angles are useful.  
                                                                                                                  
     Note that the idea of this `angle' are treated first for real normed spaces in \cite{Thuerey1} and
     \cite{Thuerey2}. 
  %     \\  \\  \\ 
  %   $$  \{ ( e^{{\bf i} \cdot \varphi} \cdot \vec{x} \, , \, \vec{y} ) \ | \  \varphi \in [ 0, 2 \pi ] \}  
  %            \ \cap \  {\cal R}_X^{\bf \bullet}  \ \neq \ \emptyset \ . $$         
 %%%%%%%%%%%%%%%%%%%%%%%%%%%%%%%%%%%%%%%%%%%%%%%%%%%%%%%%%%%%%%%%%%%%%%%%%%%%%%%%%%%%%%%%%%%%%%%%%
  %  \newpage
  \section{General Definitions} 
	            \ \ \  \ \ \ \       
    Let    \ $ X $ =  ($ X , \tau$)    be an  arbitrary complex vector  space,  that  means  that 
        the  vector space   $ X $  is  provided with a topology  $ \tau $ such that  the addition  of two vectors
        and  the  multiplication  with complex numbers  are continuous.  Further let    $\| \cdot \|$   denote a 
        {\it norm} on $ X $, i.e. there is a continuous  map  
        $\| \cdot \|$: \ $ X \longrightarrow$  ${\mathbb R}^{+} \cup \{ 0 \}$  which fulfils the following axioms \
  %   Let the pair \ $ (X, \| \cdot \|) $ \  be an  arbitrary complex normed vector  space,  i.e. the map    
  %  $\| \cdot \|$ \ has the properties 
              $\|z \cdot  \vec{x}\|$ =   $|z|$ $ \cdot  \|\vec{x}\|$ \ (`absolute homogeneity'),   
              \ $\|\vec{x}+\vec{y}\| \leq  \|\vec{x}\|  + \|\vec{y}\| $ (`triangle \ inequality'), 
              and \ $\|\vec{x}\| = 0 $ only for $ \vec{x}  = \vec{0}$ (`positive definiteness'),
              \ for \  $\vec{x}, \vec{y} \in X$  and $ z \in {\mathbb C}$.  
 
        Let  $ < . \:  | \: . > \ : \  X^{2} \longrightarrow   {\mathbb C} $    be a continuous 
	      map  from  the product space  $ X \times X $   into the field   $ {\mathbb C} $. \ Such a map is called  a
	      {\it product}.  We notice the following four conditions.  \\
%%%%%%%%%%%%%%%%%%%%%%%%%%%%%%%%%%%%%%%%%%%%%%%%%%%%%%%%%%%%%%%%%%%%%%%%%%%%%%%%%%%%%%%%%%%%%%%%%%%%%%%
  %    \newpage   $ { }  $  \\
	 $\overline{(1)}$: \ \   For  all  $ \it z  \in  {\mathbb C}$   and   
                     $\vec{x}, \vec{y}   \in X $     it holds  \
                     $ < z \cdot \vec{x} \ | \ \vec{y}> \ = \ z \cdot <\vec{x} \ | \ \vec{y}> $    
                     \  $   { \ } $  {  $ \ $ }  $ { }$       \hfill     (`homogeneity'),   \\
   $\overline{(2)}$: \  \   for  all   $\vec{x} , \vec{y} \in X$   it holds \  
                     $ <\vec{x} \ | \ \vec{y}>  \  = \ \overline{<\vec{y} \ | \ \vec{x}>} $   \hfill \quad
                                                                 (`conjugate symmetry'), \\    
   $\overline{(3)}$: \ \  for \ $ \vec{x} \in X, \ \vec{x} \neq \vec{0} $ \ we have a real number \
                     $ <\vec{x} \ | \ \vec{x}> \ \ > \ 0 $, 
 %                    and \ $  <\vec{0} \ | \ \vec{0}> \  =  0  $ \  % {  if and only if} \  $ \vec{x} = \vec{0}$  \\
                     $ { } $  \hfill   (`positive definiteness'), \\     
 %  $\overline{(4)}$: \ \  $  <\vec{x} \ | \ \vec{x}> \  =  0  $ \ {  if and only if} \  $ \vec{x} = \vec{0}$  
 %                                                     \hfill  (`definiteness'), \\
   $\overline{(4)}$: \  \  for all \  $\vec{x} , \vec{y} ,  \vec{z} \in X$  it holds \  
                     $<\vec{x} \: | \: \vec{y}+\vec{z}> \: = \:  <\vec{x} \ | \ \vec{y}> + <\vec{x} \ | \ \vec{z}>$  \\
                     $ { }$  \hfill  (`linearity in the second component'). 
   $ \begin{array}[ ]{l}                                    
             \text{If} \   < . \: | \: . >   \  \text{fulfils all properties} \
             \overline{(1)},\overline{(2)},\overline{(3)}, \overline{(4)}, \ 
             \text{the map} \ < . \: | \: . >   \text {is  an {\it inner  product}  on  {\it X}}.    
      \end{array}    $   \\
  We call the pair \ $ (X , < . \: | \: . >)$ \    %  a  {\it homogeneous product vector space},  or 
            a  {\it complex inner product space}.       
        
   It is well known that in a complex normed space   $ ( X , \| \cdot \|) $ its norm   $ \| \cdot \| $ generates 
   an inner product \ $< . \: | \: . >$ \ by equation \eqref{allererste definition} if and only if it holds  the 
   following \ {\it parallelogram identity}:  \\     %%   $\widehat{(1)}$:
   \centerline{ For   $\vec{x} , \vec{y} \in X$ \ we have  the equation \ 
   $\|\vec{x}+\vec{y}\|^{2} + \|\vec{x}-\vec{y}\|^{2} = 2 \cdot \left(\|\vec{x}\|^{2} + \|\vec{y}\|^{2} \right)$.}  \\ 
  %        $ { }  $   \hfill (`parallelogram identity').      \\  
  %%%%%%%%%%%%%%%%%%%%%%%%%%%%%%%%%%%%%%%%%%%%%%%%%%%%%%%%%%%%%%%%%%%%%%%%%%%%%%%%%%%%%%%%%%%%%%%%%%%%                    %                                        
  % Let  $\| \cdot \|$ denote a norm on a vector space $ X$. \ We define two closed subsets of  $ X$. \\
  % \centerline { $ {\bf S} \ := {\bf S}_{(X,\| \cdot \|)} \ :=  \ \{ \: \vec{x} \in X \ |  \ \|\vec{x}\| = 1 \: \}$, \
  %         the {\it unit sphere} of $X$, } \\
  % \centerline { $ {\bf B} \ := {\bf B}_{(X,\| \cdot \|)} \ :=  \ \{ \: \vec{x} \in  X \ |  \ \|\vec{x}\| \leq 1 \: \}
  %         $, \ the {\it unit ball} of $X$. }

     Assume that the complex vector space \ $X$ \ is provided with a norm   $ \| \cdot \|$, and further there is  
           a  product  $ < . \: | \: . >: X \times X \rightarrow {\mathbb C}  $. 
           We say that the triple    $( X , \| \cdot \| ,  < . \: | \: . > ) $   
           satisfies the {\it Cauchy-Schwarz-Bunjakowsky Inequality} or {\sf CSB} inequality   if and only if \
           for  $ \vec{x}, \vec{y} \in X $   there is the inequality 
    $$                         |< \vec{x} \: | \: \vec{y} >| \ \leq \ \|\vec{x}\| \cdot \|\vec{y}\| \; .   $$ 
                                          
    In the following section we need some complex valued functions like the cosine, sine, arccosine, arcsine,
    et cetera. We abbreviate them by $ \cos, \, \sin, \, \arccos, \, \arcsin $.   \\
      
   In the introduction we already defined two subset   ${\cal B}$ and  ${\cal A}$ of the complex plane ${\mathbb C}$,
   and we asserted that there are two homeomorphisms 
   $ \cos: \  {\cal A} \stackrel{\cong}{\longrightarrow} {\cal B} \ \ \ \text{and} \ \ 
	         \arccos: \ {\cal B} \stackrel{\cong}{\longrightarrow} {\cal A} $. Note that ${\cal A}$ contains only
	 inner points except two boundary points  $0$ and $\pi$, while ${\cal B}$ has $1$ and $-1$.   
	                                                       
%	 For the coming computations it is necessary to know the exact definitions of these functions. 
	 In the next definition we express the complex functions in detail by known real valued
	 $ \cos, \ \sin, \ \arccos $, arcosh, $\log $ and $ \exp $ functions.  
	                   
	 Recall the three values of the signum function,  $ {\rm sgn}(0) = 0 $, ${\rm sgn}(x) = 1 $ 
	 for positive real numbers $x$, and $ {\rm sgn}(x) = -1 $ for negative $x$.  We abbreviate the exponential
	 function by  $e^{x} := \exp(x)$. Note that in the next section we prove explicitly that the defined  arccosine
	 is truly the inverse function of the cosine.   
	  \begin{definition}   \label{die dritte definition}    \quad    
	       For a number \ $ z = a + {\bf i} \cdot b  \in {\mathbb C} $ \ its complex cosine and sine can be defined 
         explicitly by 
	  \begin{align*}   % \label{definition cosine and sine}
          \cos\left(a + {\bf i} \cdot b \right) \ & := \  \frac{1}{2} \cdot 
          \left[ \cos\left(a\right) \cdot \left(e^{b} + \frac{1}{e^{b}} \right) - {\bf i} \cdot
          \sin\left(a\right) \cdot \left(e^{b} - \frac{1}{e^{b}} \right) \right] \ ,  \\
     %     \ = \ \cos(a) \cdot \cosh(b) - {\bf i} \cdot \sin(a) \cdot \sinh(b),           \\  
           \sin\left(a + {\bf i} \cdot b\right) \ & := \  \frac{1}{2} \cdot 
          \left[ \sin\left(a\right) \cdot \left(e^{b} + \frac{1}{e^{b}} \right) + {\bf i} \cdot
          \cos\left(a\right) \cdot \left(e^{b} - \frac{1}{e^{b}} \right) \right]  \; .  
    \end{align*}   
         For a shorter presentation we can use the real hyperbolic cosine and hyperbolic sine. Their abbreviations are 
         the symbols \ $\cosh$ \ and \ $\sinh$, they are defined by  
      $$ \cosh(x) :=  \frac{1}{2} \cdot  \left(e^{x} + \frac{1}{e^{x}} \right) \ \ \text{and} \ \
         \sinh(x) :=  \frac{1}{2} \cdot  \left(e^{x} - \frac{1}{e^{x}} \right) , \ \text{for} \ x \in {\mathbb R}. $$  
       For \ $ r + {\bf i} \cdot s \in {\cal B} $ \ the functions  arcsine and  arccosine are
    \begin{align*}  % \label{definition  arccosine and  arcsine}
         \arcsin(r + {\bf i} \cdot s) \ &  :=  \ \frac{1}{2} \cdot 
         \left[ \ {\rm sgn}(r) \cdot \arccos \left( {\mathsf{G}}_- \right) \ + \ {\bf i} \cdot {\rm sgn}(s)
                   \cdot {\rm arcosh} \left( {\mathsf{G}}_+  \right) \ \right] , \ \ \text{and}     \\
         \arccos(r + {\bf i} \cdot s) \ &  := \ \frac{\pi}{2} \ - \arcsin(r + {\bf i} \cdot s) . 
    \end{align*}              
       Here we use the abbreviations 
     $$ {\mathsf{G}}_- := \sqrt{(r^{2}+s^{2}-1)^{2} + 4 \cdot s^{2}} - \left( r^{2}+s^{2} \right) , \ \text{and} \   
        {\mathsf{G}}_+ := \sqrt{(r^{2}+s^{2}-1)^{2} + 4 \cdot s^{2}} + \left( r^{2}+s^{2} \right) ,   $$
        and recall that the symbol $ {\rm arcosh}$ means the real area hyperbolic cosine, which is the inverse 
	      of the real hyperbolic cosine, 
     $$ {\rm arcosh}(x) := \log \left( x + \sqrt{x^{2} - 1} \right) \ \text{for real numbers} \ x \geq 1 \, .   $$  
        $ { }  $  \hfill     $\Box$  
    \end{definition} 
        We mention a few well-known consequences.  \\
        We assume the real cosine and sine functions and the complex exp function, all defined by its power series. 
        Since it holds $  e^{ {\bf i} \cdot r} = \cos{(r)} + {\bf i} \cdot \sin{(r)} $ for real numbers $ r $,
        with the above Definition \eqref{die dritte definition} of the complex sine and cosine % functions  
        we can deduce  Euler's formula \ $  e^{ {\bf i} \cdot z} = \cos{(z)} + {\bf i} \cdot \sin{(z)} $ 
        for all $ z \in {\mathbb C} $.  After that it is easy to prove the  identities  
        $$   \cos\left( z \right) = \frac{1}{2} \cdot \left[  e^{ {\bf i} \cdot z} + e^{ -{\bf i} \cdot z} \right] \ 
        \ \text{and} \ \   \sin\left( z \right)  =  
        \frac{1}{2 \cdot {\bf i}} \cdot \left[  e^{ {\bf i} \cdot z} - e^{ -{\bf i} \cdot z} \right] \; .  $$  
        We notice the equations $ {\rm arcosh} \circ \cosh(x) = x$  for real $ x \geq 0 $, and 
        $ \cosh \circ {\rm arcosh}(x) = x$  for $ x \geq 1 $. Further, the complex cosine can be written as \ 
        $ \cos\left(a + {\bf i} \cdot b \right) = \cos(a) \cdot \cosh(b) - 
                           {\bf i} \cdot \sin(a) \cdot \sinh(b)$,  \  while the complex sine function is \ 
        $\sin\left(a + {\bf i} \cdot b \right) = \sin(a) \cdot \cosh(b) + {\bf i} \cdot \cos(a) \cdot \sinh(b) $. 
  %     Let    $ A $   be an arbitrary subset  of a  real vector space   $ X $.    
  %          Let $ A $    has the  property  that 
  %          for  arbitrary   $ \vec{x} , \vec{y}  \in A $   and for  every    $ 0 \leq t \leq 1 $   it holds   
	%          $ t \cdot \vec{x} + (1-t) \cdot \vec{y} \in A. $    Such a  set $ A $ is called  {\it  convex}.
	%          The unit ball   $ {\bf B} $   in a normed space is convex because of the triangle inequality.  
	%  	
	%%%%%%%%%%%%%%%%%%%%%%%%%%%%%%%%%%%%%%%%%%%%%%%%%%%%%%%%%%%%%%%%%%%%%%%%%%%%%%%%%%%%%%%%%%%%%%%%%%%%                    %      A  convex set  $ A $ is called {\it strictly convex} if and only if for each number $ 0 < t < 1 $   
  %          it holds that the linear combination   $ t \cdot \vec{x} + (1-t) \cdot \vec{y}$   lies in the interior of 
  %          $ A $, for all distinct vectors $ \vec{x},  \vec{y} \in A $.   
  %                                                                                  
	%       For two real numbers   $ a < b $   the term   $[a, b]$   means the closed interval of   $a$ and $b$, \
  %       while  $(a,b)$   means the pair of two numbers or the open interval between  $a$ and $b$.    
  %  \newpage
  %%%%%%%%%%%%%%%%%%%%%%%%%%%%%%%%%%%%%%%%%%%%%%%%%%%%%%%%%%%%%%%%%%%%%%%%%%%%%%%%%%%%%%%%%%%%%%%%%%%%%%  
  % \newpage  
   \section{Complex Normed Spaces}    \label{section three}
     First we prove that the cosine and  arccosine as they are written in Definition \eqref{die dritte definition} 
     are actually inverse functions. 
  %   We can find an equivalent definition of the  arccosine in \cite{Boisvert/Clark/Olver/Lozier}.  
     \begin{lemma}    \label{lemma cosine circ  arccosine}
          For all $ z \in {\cal B} $ it holds \ $ \cos \circ \arccos (z) = z $.  
     \end{lemma} 
     \begin{proof}
          We take an arbitrary element \ $ z = r + {\bf i} \cdot s \in {\cal B} $. We use the abbreviations 
          $ {\mathsf{G}}_- $ and $ {\mathsf{G}}_+ $ of Definition \eqref{die dritte definition}, and with easy
          calculations we get 
      $$   (1 - {\mathsf{G}}_-) \cdot ({\mathsf{G}}_+ + 1) = 4 \cdot r^{2} \, , \ \ \text{and} \ \   
           (1 + {\mathsf{G}}_-) \cdot ({\mathsf{G}}_+ - 1) = 4 \cdot s^{2} \, .                     $$ 
        We assume for  $ z = r + {\bf i} \cdot s \in {\cal B} $ \ that both $ r $ and $ s $ are positive. 
        We use Definition \eqref{die dritte definition} and elementary calculus, and we compute 
        \begin{align*} 
             & \cos \circ \arccos \, (z) \ = \ \cos \ [ \arccos (r + {\bf i} \cdot s ) ]  %  \\  &
             \ = \ \cos\left[ \ \frac{\pi}{2} \ - \ \frac{1}{2} \cdot 
                   \arccos \left( {\mathsf{G}}_- \right) \ - \ {\bf i} \cdot \frac{1}{2} \cdot 
                   {\rm arcosh} \left( {\mathsf{G}}_+  \right) \ \right]  \\
             & = \ \cos\left( \frac{\pi}{2} \ - \ \frac{1}{2} \cdot \arccos ( {\mathsf{G}}_- )  \right)
                   \cdot \cosh \left( - \, \frac{1}{2} \cdot {\rm arcosh} ({\mathsf{G}}_+) \right) \\
             & \quad \     \ - \ {\bf i} \cdot \   
                   \sin\left( \frac{\pi}{2} \ - \ \frac{1}{2} \cdot \arccos ( {\mathsf{G}}_- )  \right) 
                   \cdot \sinh \left( - \, \frac{1}{2} \cdot {\rm arcosh} ({\mathsf{G}}_+) \right)  
                   \qquad (\text{note} \ \sinh(-x) = -\sinh(x))       \\
             & = \ \sin\left( \frac{1}{2} \cdot \arccos ( {\mathsf{G}}_- )  \right)
                   \cdot \cosh \left( \frac{1}{2} \cdot {\rm arcosh} ({\mathsf{G}}_+) \right) % \  \quad & 
                    + \ {\bf i} \cdot \   
                   \cos\left( \frac{1}{2} \cdot \arccos ( {\mathsf{G}}_- )  \right) 
                   \cdot \sinh \left( \frac{1}{2} \cdot {\rm arcosh} ({\mathsf{G}}_+) \right)  \\      
             & = \ \sqrt{\frac{1}{2} \cdot (1 - {\mathsf{G}}_-)}  \cdot
                   \sqrt{\frac{1}{2} \cdot ({\mathsf{G}}_+ + 1)} \ + \ {\bf i} \cdot \   
                   \sqrt{\frac{1}{2} \cdot (1 + {\mathsf{G}}_-)}  \cdot
                   \sqrt{\frac{1}{2} \cdot ({\mathsf{G}}_+ - 1)}  \\ 
             & = \ \frac{1}{2} \cdot  \sqrt{(1 - {\mathsf{G}}_-) \cdot ({\mathsf{G}}_+ + 1)} \ + \ {\bf i} \cdot \                        \frac{1}{2} \cdot  \sqrt{(1 + {\mathsf{G}}_-) \cdot ({\mathsf{G}}_+ - 1)}
                   \ = \  r + \ {\bf i} \cdot s \ = \ z .          
        \end{align*}   
        The other three cases are \ $ r \cdot s = 0 $, \ $ r \cdot s < 0 $, and both $ r $ and $ s $ are negative.
        By noting the signs they work in the same manner, and the lemma is established.      
     \end{proof} 
     \begin{proposition}   \it      \label{proposition zwei}        
          We have two homeomorphisms 
	        \begin{align*}
	              & \cos: \ {\cal A} \stackrel{\cong}{\longrightarrow} {\cal B} \ \ \ \text{and} \ \ 
	              \arccos: \ {\cal B} \stackrel{\cong}{\longrightarrow} {\cal A} \, ,  \ \
	              \text{and we get two identities}  \\  
	              & \cos \circ \arccos = id_{{\cal B}} \ \ \ \text{and} \ \ \arccos \circ \cos = id_{{\cal A}} \, .
	        \end{align*}  
	  \end{proposition}
	  \begin{proof}   
	        By Definition  \eqref{die dritte definition}  we have two continuous maps, \
	        $ \cos: \ {\cal A} \rightarrow {\cal B}$ \ and \ $ \arccos: \ {\cal B} \rightarrow {\cal A}$.   
	        We just proved \ $ \cos \circ \arccos (z) = z $ for all $ z \in {\cal B} $.  
	        This shows that the  arccosine function is injective on its domain ${\cal B}$, and the cosine function 
	        is surjective on its codomain ${\cal B}$. If we consider all possible cases for $ a $ and $ b $, for  
	        $  a + {\bf i} \cdot b  \in  {\cal A} $, we see that the cosine is injective on the domain ${\cal A}$. 
	        It follows that the cosine function is a bijective map 
	        ${\cal A} \rightarrow {\cal B}$.  We have $ \cos \circ \arccos = id_{\cal B} $. Therefore it holds \    
	        $ \arccos = \cos^{-1} \circ \ id_{\cal B} = \cos^{-1} $, and we get that the  arccosine is also a bijective
	        map, and it is the inverse of the cosine.   %  ${\cal B} \rightarrow {\cal A}$.  
	        Because both functions are continuous, they are both homeomorphisms.
	%       The last point is the surjectivity on ${\cal A}$ of the  arccosine function.  \\
	%          \\    See also  \cite{Boisvert/Clark/Olver/Lozier}, p.120,  where an alternative definition of
	%          the  arccosine is given. (For an internet version see \cite{NIST}.) 
	%          Some efforts are needed to prove the equivalence of both definitions.       \\
	%          (  Vorzeichen  ????????????????????????????????????????????  L\"oschen !!!!!!!!!!!!!!)  
	  \end{proof}       
    We describe that the cosine and  arccosine functions         %%    \ $ \cos $ and $ \arccos $ 
        have a `central symmetrical behavior', they respect the `center points' \ 
        $ \pi/2 $ of ${\cal A}$ \ and \ $ 0 $ of ${\cal B}$, respectively. 
    Note that for each complex number $ z $ we can write $ z = \frac{\pi}{2} + a + {\bf i} \cdot b $, 
	      with a suitable real number $  a $. It means that the real part of $ z $ is \ $ \frac{\pi}{2} + a $. 
   %     The next proposition will  make this  clearer.     %    More informative is the next proposition.
    \begin{proposition}  \label{proposition eins} 
         For each complex number $ z $ in ${\cal A}$ we write $ z = \frac{\pi}{2} + a + {\bf i} \cdot b $, 
         with a suitable real number $ - \frac{\pi}{2} \leq a \leq  \frac{\pi}{2}$. \ If \ 
         $ \cos\left(\frac{\pi}{2} + a + {\bf i} \cdot b\right) = r + {\bf i} \cdot s$ \ it follows 
         $$ \cos\left(\frac{\pi}{2} - a - {\bf i} \cdot b\right) \ = \ -r - {\bf i} \cdot s \ .  $$   
         Correspondingly,  for a number $r + {\bf i} \cdot s \in {\cal B}$ \ with \
         $ \arccos\left( r + {\bf i} \cdot s \right) = \frac{\pi}{2} + a + {\bf i} \cdot b$ \  it holds 
         $$ \arccos\left( - r - {\bf i} \cdot s \right) \ = \ \frac{\pi}{2} - a - {\bf i} \cdot b \ .  $$   
    \end{proposition}  
     \begin{proof}
          For the  arccosine we see this immediately from Definition \eqref{die dritte definition}.
          Note that both the real  arccosine and the real area hyperbolic cosine have a non-negative image.  
          Since the cosine is the inverse function 
          of the  arccosine, it must act as it is claimed in the proposition.
     \end{proof}      
       Let $ ( X, \| \cdot \| ) $  be a complex normed vector space. In Definition \eqref{die erste definition}
	     we defined a continuous product \   $ < . \: | \: . > $ \ on $ X $. In the introduction we already mentioned 
	     that this is the inner product in the case that the norm  $ \| \cdot \| $ generates an inner product by  
	     equation \eqref{allererste definition}. Also we noticed some properties of this product. We discuss them now. 
	     \begin{proposition}  \label{proposition drei} 
	      For all vectors \ $\vec{x}, \vec{y} \in (X, \| \cdot \|) $ and for real numbers \ $ r $ \ 
	      the product \   $ < . \: | \: . > $ \ has the following properties  \\
	      { \rm(a)} \ $ < \vec{x} | \vec{y}>  \ = \ \overline{<\vec{y} | \vec{x}>} $  \hfill (conjugate symmetry), \\
	      { \rm(b)} \ $ <\vec{x} \ | \ \vec{x}> \ \ \geq \ 0 $, \ and  \  
                     $  <\vec{x} \ | \ \vec{x}> \  =  0  $ \ {only for} \  $ \vec{x} = \vec{0}$ 
	                                                            \hfill   (positive definiteness), \\
	      { \rm(c)} \ $  < r \cdot \vec{x} \, | \, \vec{y} > \ = \ r \cdot <  \vec{x} \, | \, \vec{y} > $ 
	                                                        \hfill  (homogeneity for real numbers), \\
	      { \rm(d)} \ \ $  < r \cdot {\bf i} \cdot \vec{x} \, | \, \vec{y} > \ = \
	                        r \cdot {\bf i} \cdot <  \vec{x} \, | \, \vec{y} > $   
	                                 \hfill   (homogeneity for pure imaginary numbers),  \\
	      { \rm(e)}\ \ $ \| \vec{x} \| = \sqrt{ < \vec{x} | \vec{x}> }  $
	                                             \hfill   (the norm can be expressed by the product).      
 %%       $  { } $  \hfill    for $  \vec{x},  \vec{y} \in X , \ r \in {\mathbb R}$.  (so  klappt's !!!
       \end{proposition} 
       \begin{proof}
           The proofs for {\rm(a)} and  {\rm(b)} are trivial. The point {\rm(c)} is trivial  for positive $r$.
           The proof of $  < - \vec{x} \, | \, \vec{y} >  = - <  \vec{x} \, | \, \vec{y} > $ is easy, 
           and  {\rm(c)} follows immediately.  The point {\rm(d)} is similar to {\rm(c)}, and  {\rm(e)} is clear. 
       \end{proof}
  %     For the coming considerations it  is very useful to introduce a few abbreviations.                        
  %     For arbitrary vectors  $ \vec{x}, \vec{y} \neq \vec{0}$ % (hence  $ \|\vec{x}\| \cdot \|\vec{y}\|  \neq  0 $)   
  %     we define four non-negative real numbers $ {\bf s} $, $ {\bf d} $,  $ {\bf s_{ i} } $, and $ {\bf d_{ i}} $. Let
  %  \begin{align*}   
  %     {\bf s}  \ :=  \  {\bf s}( \vec{x}, \vec{y}) 
  %     \ & :=  \  \left\| \frac{\vec{x}}{\|\vec{x}\|}   +  \frac{\vec{y}}{\|\vec{y}\|} \right\| , \quad 
  %     {\bf d} \ := \  {\bf d}( \vec{x}, \vec{y})  
  %     \ :=  \   \ \left\| \frac{\vec{x}}{\|\vec{x}\|}   -  \frac{\vec{y}}{\|\vec{y}\|} \right\| \ ,      \\
  %     \text{as well} \ \ \ \
  %      {\bf s_{ i}}  \ :=  \  {\bf s_{ i}}( \vec{x}, \vec{y}) 
  %     \ & :=  \  \left\| \frac{\vec{x}}{\|\vec{x}\|}   + {\bf i} \cdot \frac{\vec{y}}{\|\vec{y}\|} \right\| , 
  %     \ \ \text{and} \ \
  %    {\bf d_{ i}} \ := \  {\bf d_{ i}}( \vec{x}, \vec{y})  
  %     \ :=  \   \ \left\| \frac{\vec{x}}{\|\vec{x}\|} - {\bf i} \cdot \frac{\vec{y}}{\|\vec{y}\|} \right\| \ .    
  %   \end{align*}   
  %    All  defined four variables depend on two vectors \ $ \vec{x}, \vec{y} \neq \vec{0} $.  \     
  %    With these abbreviations Definition \eqref{die zweite definition} of the angle $ \angle(\vec{x}, \vec{y}) $ 
  %    can be written as 
       In Definition \eqref{die zweite definition} we defined the angle $ \angle(\vec{x}, \vec{y}) $, we wrote \ \
       $\angle(\vec{x}, \vec{y}) \ := \ 
                     \arccos{ \left( \frac{< \vec{x} \: | \: \vec{y} > } { \|\vec{x}\| \cdot \|\vec{y}\|} \right)}$
       \begin{align*}   
         %          \arccos{ \left( \frac{< \vec{x} \: | \: \vec{y} > } { \|\vec{x}\| \cdot \|\vec{y}\| } \right) } \\
         %          \ = \  \arccos{ \left( \frac{1}{4} \cdot 
         %          \left[ {\bf s}^{2}  - {\bf d}^{2} +
         %          {\bf i} \cdot  \left( {\bf s_{ i}}^{2} - {\bf d_{ i}}^{2}  \right)  \right] \right)  }  \\
                    \  = \  \arccos{ \left( \frac{1}{4} \cdot 
                   \left[ \left\| \frac{\vec{x}}{\|\vec{x}\|}   +  \frac{\vec{y}}{\|\vec{y}\|} \right\|^{2}   - 
                   \left\| \frac{\vec{x}}{\|\vec{x}\|} - \frac{\vec{y}}{\|\vec{y}\|} \right\|^{2} +
                   {\bf i} \cdot  \left( \left\| \frac{\vec{x}}{\|\vec{x}\|}  +  
                   {\bf i} \cdot \frac{ \vec{y}}{\|\vec{y}\|} \right\|^{2} - \left\| \frac{\vec{x}}{\|\vec{x}\|} -
                   {\bf i} \cdot \frac{ \vec{y}}{\|\vec{y}\|} \right\|^{2}  \right)  \right] \right)  }   \ . \                 \end{align*} 
     
   %     In the introduction we defined for arbitrary vectors  $ \vec{x}, \vec{y} \neq \vec{0} $ a complex number 
   %     $\angle(\vec{x}, \vec{y})$ in  ${\cal A}$, we defined in Definition (\eqref{die zweite definition}) \\  
   %         \centerline{  $ \angle(\vec{x}, \vec{y}) \ = \ $ } 
	 % \begin{align*}  
   %                
   %  \end{align*}        
       This complex number $ \angle (\vec{x}, \vec{y}) $ is called the {\it  complex angle} 
	     of the pair $ (\vec{x}, \vec{y}) $.  %     \hfill     $\Box$ 
	 %    Note that both the real part and the imaginary part of \  $ \cos(\angle(\vec{x}, \vec{y})) $ \ 
	 %    are in the interval $  [ -1, 1 ] $. 
	 %    the angle is always defined, since by the triangle inequality we get that the real part 
	 %    of $ \frac{1}{4} \cdot  \left[ {\bf s}^{2}  - {\bf d}^{2} +
   %                {\bf i} \cdot  \left( {\bf s_{ i}}^{2} - {\bf d_{ i}}^{2}  \right)  \right] $  
   %    We repeat this important fact in a lemma.                                     
	 %%%%%%%%%%%%%%%%%%%%%%%%%%%%%%%%%%%%%%%%%%%%%%%%%%%%%%%%%%%%%%%%%%%%%%%%%%%%%%%%%%%%%%%%%%%%%%%%%%%%%%%%%%         
    \begin{lemma}   \label{lemma eins}
    For a pair $ \vec{x}, \vec{y} \neq \vec{0} $ in a complex normed space  $ ( X, \| \cdot \| )$ it holds that 
      %   $ \frac{1}{4} \cdot  \left[ {\bf s}^{2}  - {\bf d}^{2} +
      %             {\bf i} \cdot  \left( {\bf s_{ i}}^{2} - {\bf d_{ i}}^{2}  \right)  \right] $.   
         both the real part     %  \ $ \frac{1}{4} \cdot \left( {\bf s}^{2}  - {\bf d}^{2} \right)  $ 
         and the imaginary part of \ $ \cos(\angle(\vec{x}, \vec{y})) $ \ are in the interval $ [ -1, 1 ] $.  
     %  \ $ \frac{1}{4} \cdot  \left( {\bf s_{ i}}^{2} - {\bf d_{ i}}^{2}  \right)  $
     % \ $ \frac{1}{4} \cdot ({\bf s_{ i}}^{2} - {\bf d_{ i}}^{2})$ \ % 
    \end{lemma}      
    \begin{proof} 
         The lemma can be proven easily with the triangle inequality.
    \end{proof}   
    \begin{corollary} \label{corollary always angle}
         Lemma \eqref{lemma eins} means that \
         $ \left\{ \cos(\angle(\vec{x}, \vec{y})) \ | \ \vec{x}, \vec{y} \in X\backslash \{\vec{0}\} \right\} $
         is a subset of the `complex square' \
         $ {\cal SQ} :=  \{ r + {\bf i} \cdot s \in {\mathbb C} \ | \ -1 \leq r, s \leq +1 \}$,
         i.e. $\left| \cos(\angle(\vec{x}, \vec{y})) \right| \ \leq \ \sqrt{2}$. \
         It follows  $  \cos(\angle(\vec{x}, \vec{y})) \in {\cal B}$ and  $ \angle(\vec{x}, \vec{y}) \in {\cal A}$,
         i.e. the angle  $ \angle (\vec{x}, \vec{y}) $ is defined for each pair \ $\vec{x}, \vec{y} \neq \vec{0} $.  
    \end{corollary} 
    \begin{corollary} \label{corollary hexagon}
         The values of the `angles' are from the image $ \arccos({\cal SQ}) $, which forms a symmetric hexagon in
         $ {\cal A} $.  Its center is $ \pi/2$, two corners are $ 0 $ and $ \pi $.  A third is, for instance, \\
         \centerline{ $ \arccos(1 + {\bf i}) \ = \ (\pi/2) - (1/2) \cdot \left[ \arccos(\sqrt{5}-2) + {\bf i} \cdot  
         \log \left(\sqrt{5}+2 + 2 \cdot \sqrt{\sqrt{5}+2} \right)\right] \ \approx \ 0.90 - {\bf i} \cdot 1.1 $.} \\
         We get the other three corners easily with the aid of Proposition \eqref{proposition eins}. 
         
    \end{corollary}   
         Now we prove that a stronger property fails.
   \begin{theorem} \label{conjecture csb}
     %    Let  $ ( X, \| \cdot \| )$ be a complex normed space  $ ( X, \| \cdot \| )$. 
         Generally the {\sf CSB }inequality is not fulfilled. To show this we construct examples of 
         complex normed spaces $ X $ with elements \ $ \vec{x} , \vec{y} $ \  such that  
        $$ |< \vec{x} \: | \: \vec{y} >   |  \ > \ \|\vec{x}\| \cdot \|\vec{y}\|  \ \ \text{or, in other words,} \ \
        1 \ < \   \left| \cos(\angle(\vec{x}, \vec{y})) \right| \ .            $$ 
   \end{theorem}  
   \begin{proof}
        The construction needs some preperations. First, let \ $ {\mathsf S} \subset X  $ \
        be any subset of the complex vector space $ X \neq \{ \vec{0} \} $. 
        Note that in the beginning $ X $ carries no norm.  
    \begin{definition}  
        If \ $ {\mathsf S}  $ \ has the property that for any vector \
        $ \vec{x} \in {\mathsf S}  , \ \vec{x} \neq  \vec{0} $ \ there is a real  number \ $ \omega > 0 $ \ such that \
        $ \overline{\omega} \cdot \vec{x} \notin  {\mathsf S} $ \ for all \ $  \overline{\omega} \geq \omega $, \ 
        we call the set  $ {\mathsf S} $ { \it `locally bounded'}.  \\
          If \ $ {\mathsf S} $ \ has the  property that for any vector \
        $ \vec{x} \in X , \ \vec{x} \neq  \vec{0} $ \ there is a real number \ $ \varepsilon > 0 $ \ such that \
        $ \varepsilon \cdot \vec{x} \in  {\mathsf S} $, \ the set $ {\mathsf S} $ is called { \it `absorbing'}.    
    \end{definition}     
    \begin{definition}  \label{minkowski property} 
        We say that \ $ {\mathsf S} \subset X $ \ has the `${\mathsf MINKOWSKI }$-property' \ if and only if \
        $ {\mathsf S} $ \ is both { \it locally bounded} and { \it absorbing}. \ In this case 
        we can   define a positive definite functional \ $ \| \cdot \|_{ \mathsf S} $ on $ X $, i.e. 
        $ \| \cdot \|_{ \mathsf S} :  X \rightarrow  {\mathbb R}^{+} \cup \{ 0 \} $, the so-called \
        `${\mathsf MINKOWSKI }$-functional'.  \  We define  
    \begin{align}    
         \|\vec{x}\|_{ \mathsf S} \ := \   \inf \left\{ \eta > 0 \ | \ \eta^{-1} \cdot \vec{x} \in {\mathsf S}
         \right\} \ \text{for} \ \ \vec{x} \neq \vec{0} , \ \ \text{and} \ \ \|\vec{0}\|_{ \mathsf S}  :=  0 \ . 
    \end{align}   
    \end{definition}     
       If in addition the set ${\mathsf S}$ is {\it balanced}
       (i.e. for $ \vec{x} \in {\mathsf S}$ and $ \phi \in [ 0, 2 \, \pi ] $ it follows 
       $ e^{{\bf i} \cdot \phi}  \cdot \vec{x} \in {\mathsf S}$) \  and {\it convex} (i.e. for two elements \ 
       $ \vec{x}, \vec{y} \in {\mathsf S}$  the line segment $ {\mathsf SEG} $  between 
       $ \vec{x}$  and $ \vec{y} $,  i.e. the set   \\
    \centerline  {$  {\mathsf SEG} := \left\{ s \cdot \vec{x} + t \cdot \vec{y} \ | \ 0 \leq s, t \leq 1 , \ \text{and}
                                          \ s + t = 1 \right\}, $       }    \\    
       is a subset of ${\mathsf S}$),  the just constructed  ${\mathsf MINKOWSKI }$-functional \
       $ \| \cdot \|_{ \mathsf S} $ \ is a norm.
       This means that the  pair \ $ \left( X, \| \cdot \|_{ \mathsf S} \right) $ \ is a complex normed vector space. \ 
       (Compare  \cite{Rudin}  or \cite{Werner}.)   
                                                                  
       We need some maps which extend subsets of $ X $. 
    \begin{definition} Let ${\sf S} \neq \emptyset $ be any subset of the complex vector space $ X $. We define 
    \begin{align*}  
     {\mathsf twist}({\sf S}) \ & := \ \left\{ e^{{\bf i} \cdot \phi}  \cdot \vec{x} \ | \ 
                         \phi \in {\mathbb R} \ \ \text{and} \ \ \vec{x} \in {\sf S} \right\}  \ ,  \\ 
     {\mathsf conv}({\sf S}) \ & := \ \left\{ \sum_{j = 1}^{k} t_j \cdot \vec{x_j} \ | \  k \in  {\mathbb N} , \
       0 \leq t_j \leq 1 , \ \vec{x_j} \in {\sf S} \ \ \text{for all} \ \ 1 \leq j \leq k , \ \ \text{and} \
       \sum_{j=1}^{k} t_j = 1      \right\} \ . % \\  
   %   {\mathsf span}({\sf S}) \ & := \ \left\{ \sum_{j = 1}^{k} z_j \cdot \vec{x_j} \ | \  k \in  {\mathbb N} , \
   %    z_j \in  {\mathbb C} , \ \vec{x_j} \in {\sf S} \ \ \text{for all} \ \ 1 \leq j \leq k    \right\} \ .   
    \end{align*}    
    We complete the definition with \ 
    $ {\mathsf twist}(\emptyset) : =  {\mathsf conv}(\emptyset) :=  \{ \vec{0} \} $. 
    \end{definition}
    We mention some facts about the two maps. All proofs are easy, and we abstain from showing them. 
    Let \ $ {\sf S} \subset X $. 
  \begin{itemize}
       \item \     ${\sf S} \ \subset \ {\mathsf twist}({\sf S}) \cap {\mathsf conv}({\sf S})$, 
   %                \cap {\mathsf span}({\sf S})$,
       \item \     $  {\mathsf twist} \circ {\mathsf twist} \ = \  {\mathsf twist} $, \  
                   $  {\mathsf conv} \circ {\mathsf conv} \ = \  {\mathsf conv} $, \ 
   %                $  {\mathsf span} \circ {\mathsf span} \ = \  {\mathsf span} $, 
       \item \     $  {\mathsf twist} \circ {\mathsf conv} ({\sf S}) \ 
                                         \subset \ {\mathsf conv} \circ {\mathsf twist} ({\sf S})  $,    
       \item \     $  {\mathsf twist} \circ {\mathsf conv} \circ {\mathsf twist} 
                                                      \ = \ {\mathsf conv} \circ {\mathsf twist} $, \
       \item \     $   {\mathsf conv}({\sf S}) $ is convex,  $ {\mathsf twist}({\sf S}) $ is balanced,  
                       ${\mathsf conv} \circ {\mathsf twist} ({\sf S}) $ is both convex and balanced.
       \item \     If $ {\sf S} $ is locally bounded, $ {\mathsf twist}({\sf S}) $ is locally bounded, too.   
  \end{itemize}  
   (The reader may look for a locally bounded set $ {\sf S} $ such that $ {\mathsf conv}({\sf S}) $ 
   is not locally bounded.) 
                                                                  
   We are constructing examples  of complex normed vector spaces which do not fulfil the {\sf CSB }inequality to prove 
   Theorem  \eqref{conjecture csb}. As the vector space we take the easiest non-trivial space, the two-dimensional 
   space \ 
   $ {\mathbb C}^{2} = \left\{ v \cdot ( 1, 0) + w \cdot ( 0, 1) \ | \ v,w \in {\mathbb C} \right\} $. \ 
   For each positive real number \ $ r $, i.e. \ $ r > 0 $,  we construct a norm \
   $ \| \cdot \|_{r} $ \ on \ $ {\mathbb C}^{2} $. \ We define 
   the finite subset \ $ {\sf S}_r $ of $ {\mathbb C}^{2}$, the set \ ${\sf S}_r $ \  contains four elements 
   $$   {\sf S}_r :=  \left\{ \left(  \begin{array}{c} 1  \\ 0  \end{array}  \right),  \ 
                              \left(  \begin{array}{c} 0  \\ 1  \end{array}  \right),  \ 
                              \left(  \begin{array}{c} r \\ -r  \end{array}  \right),  \
                              \left(  \begin{array}{c}  r \\ -{\bf i} \cdot r  \end{array}  \right) \right\} \ . $$  
      We build the convex and balanced set   
           $$ {\sf M}_r :=  {\mathsf conv} \circ {\mathsf twist} \, ({\sf S}_r) \ . $$  
    \begin{lemma} \label{lemma minkowski property}
             The set \ $ {\sf M}_r $ \ fulfils the ${\mathsf MINKOWSKI }$-property of 
             Definition \eqref{minkowski property}.
    \end{lemma} 
    \begin{proof}  Let \ $ \vec{x} \neq \vec{0} $. The subset \
       $ \{ (1,0), (0,1) \} \subset {\sf S}_r $ \ is already a basis of  $ {\mathbb C}^{2}$.
       Hence there is a representation  $ \vec{x} = v \cdot ( 1, 0) + w \cdot ( 0, 1) $ with suitable complex numbers 
       $ v, w $.    We can write this as  \
      $ \vec{x} = a \cdot e^{ {\bf i} \cdot \phi} \cdot ( 1, 0) + b \cdot e^{ {\bf i} \cdot \psi} \cdot ( 0, 1) $             with non-negative real numbers \ $ a, b $, and \ $ \phi, \psi \in [0,2 \pi ] $. \ We take the sum \
      $ \chi := a + b > 0 $, \ and we write \ 
      $$ \vec{x} =  \chi \cdot \left[ \frac{a}{\chi} \cdot e^{ {\bf i} \cdot \phi} \cdot 
         \left(  \begin{array}{c} 1  \\ 0  \end{array}  \right) 
         + \frac{b}{\chi}  \cdot e^{ {\bf i} \cdot \psi} \cdot 
         \left(  \begin{array}{c} 0  \\ 1  \end{array}  \right)   \right] \ .  $$  
      Since \ $ (a / \chi) \cdot e^{ {\bf i} \cdot \phi} \cdot ( 1, 0) 
      + (b / \chi)  \cdot e^{ {\bf i} \cdot \psi} \cdot ( 0, 1) \in  {\sf M}_r $ \ it follows 
      \ $ (1 / \chi) \cdot  \vec{x}  \in  {\sf M}_r $, i.e.  ${\sf M}_r $ is absorbing. 
                                 
      Since  \ $ {\sf S}_r $ is finite, it is locally bounded. This property will be transported by $ {\mathsf twist} $
      and in this case also by $ {\mathsf conv} $, since \ $ {\sf S}_r $ is finite. 
      Hence $ {\sf M}_r $ is locally bounded, too. The lemma is proven. 
    \end{proof}      
    Now we define for each number  $ r \in {\mathbb R}^{+}$  a norm on  ${\mathbb C}^{2}$. We use the
    ${\mathsf MINKOWSKI }$-functional of Definition  \eqref{minkowski property}. \ 
    Let for \ $ \vec{x} \in {\mathbb C}^{2}$  
   $$  \|\vec{x}\|_{r} \ := \   \inf \left\{ \eta > 0 \ | \ \eta^{-1} \cdot \vec{x} \in  {\sf M}_r 
         \right\} \ \text{for} \ \ \vec{x} \neq \vec{0} , \ \ \text{and} \ \ \|\vec{0}\|_{r}  :=  0 \ . $$   
   The set \  ${\sf M}_r$ is convex and balanced, and by Lemma \eqref{lemma minkowski property} it satisfies 
   the ${\mathsf MINKOWSKI}$-property. We get that the pair $ ({\mathbb C}^{2},  \| \cdot \|_{r}) $ is a
   normed space, please see the remark after Definition \eqref{minkowski property}.   
                                                                                               
   Now we take two vectors \ $ \vec{a} := (1,0) $ and $ \vec{b} := (0,1) $, both are unit vectors in the spaces
   $ ({\mathbb C}^{2},  \| \cdot \|_{r}) $, for all $ r \in {\mathbb R}^{+} $.  We consider the product  \                    $< \vec{a} \: | \: \vec{b} >_r := < \vec{a} \: | \: \vec{b} > $. \ 
   By Definition \eqref{die erste definition} we have  
   \begin{align}   \label{term  product complex normed spaces}
     < \vec{a} \: | \: \vec{b} >_r \ & = \ \frac{1}{4} \cdot  
                  \left[ \left\| \vec{a} +  \vec{b} \right\|_r^{2} -   \left\| \vec{a} - \vec{b} \right\|_r^{2} +
                  {\bf i} \cdot  \left( \left\| \vec{a}  + {\bf i} \cdot \vec{b} \right\|_r^{2} - 
                  \left\| \vec{a} - {\bf i} \cdot \vec{b} \right\|_r^{2} \right) \right] \\    
                  \ & = \ \frac{1}{4} \cdot  
                  \left[ \left\| (1,1) \right\|_r^{2} -   \left\| (1,-1) \right\|_r^{2} +
                  {\bf i} \cdot  \left( \left\| (1,  {\bf i} ) \right\|_r^{2} - 
                  \left\| (1,  -{\bf i} )  \right\|_r^{2} \right) \right] \ .  
    \end{align}   
    We consider the norms of four vectors \
     $ \|(1,1)\|_r, \ \|(1,-1)\|_r, \ \|(1, {\bf i})\|_r, \ \|(1,-{\bf i})\|_r $. \ Since we have convex combinations \
     $$ \frac{1}{2} \cdot   \left(  \begin{array}{c} 1  \\ 1  \end{array}  \right) 
     =  \frac{1}{2} \cdot  \left(  \begin{array}{c} 1  \\ 0  \end{array}  \right)  + 
        \frac{1}{2} \cdot  \left(  \begin{array}{c} 0  \\ 1  \end{array}  \right) \ \in {\sf M}_r  \ , \ \ \
        \frac{1}{2} \cdot  \left(  \begin{array}{c} 1  \\  {\bf i}  \end{array}  \right) 
     =  \frac{1}{2} \cdot  \left(  \begin{array}{c} 1  \\  0 \end{array}  \right)  +  
        \frac{1}{2} \cdot  \left(  \begin{array}{c} 0  \\  {\bf i}  \end{array}  \right)  \ \in {\sf M}_r \ ,  $$
     it  follows \ $ \|(1,1)\|_r = 2 = \|(1, {\bf i})\|_r $ .    
                                                                        
     By definition of the set ${\sf M}_r$ it contains two vectors $ (r , -r) , \ ( r , -{\bf i} \cdot r ) $. 
     For  $ r \geq 1/2 $ \ both are unit vectors in \ $ ({\mathbb C}^{2}, \| \cdot \|_{r}) $. \ We have  
     $ r \cdot (1,-1) = (r , -r) $ and $ r \cdot (1,-{\bf i}) = (r , -{\bf i} \cdot r) $.  \ If  $ r $ 
     is sufficient positive, i.e. $ r \geq 1/2 $, we get the two norms \
     $  \|(1,-1)\|_r =  1/r = \|(1,-{\bf i})\|_r $. \  
     Finally we get for the product \ $ < \vec{a} \: | \: \vec{b} >_r $ \ the result     
    \begin{align} 
           < \vec{a} \: | \: \vec{b} >_r \ = \ \frac{1}{4} \cdot  \left[ 2^{2} - \left(\frac{1}{r}\right)^{2} + 
           {\bf i} \cdot \left( 2^{2} - \left(\frac{1}{r}\right)^{2} \right) \right] \ ,  
    \end{align} 
    this means for the modulus    
    \begin{align} 
           \left| < \vec{a} \: | \: \vec{b} >_r \right| 
                   \ = \ \sqrt{ \left[1 - \left(\frac{1}{2 \cdot r}\right)^{2} \right]^{2} + 
                         \left[1 - \left(\frac{1}{2 \cdot r}\right)^{2} \right]^{2}} 
                   \ = \ \sqrt{2} \cdot \left[1 - \left(\frac{1}{2 \cdot r}\right)^{2} \right] \ . 
    \end{align} 
    Obviously, we get the  limit \ $ \lim_{r \rightarrow \infty} (|< \vec{a} \: | \: \vec{b} >_r|) = \sqrt{2} $. 
    \ Since \ $ \|\vec{a}\|_{r}  = 1 = \|\vec{b}\|_{r} $ \ for all \ $ r $,  this is sufficient to prove  
    Theorem \eqref{conjecture csb}. 
   \end{proof}   
   Because the expression in equation \eqref{term  product complex normed spaces} yields a characteristic property 
   of complex normed spaces, we add a definition and a proposition. 
   \begin{definition}
       Let  \ $ (X, \| \cdot \|) $ be a complex normed vector space. \ We define a positive number \ 
       $ { \cal D }(X, \| \cdot \|) $, the \ {\it `Deformation'}  of  $ (X, \| \cdot \|) $. \ Let 
        $ { \cal D }(\{ \vec{0} \}, \| \cdot \|) := 1 $, and for $ X \neq \{ \vec{0} \} $ we define    
   \begin{align*}   { \cal D }(X, \| \cdot \|) \ := \   
      \sup \ \left\{ \ \text{the modulus of} \ \left(< \vec{a} \: | \: \vec{b} >\right) \ | \ \vec{a}, \vec{b} \ 
      \text{are unit vectors in} \ \left(X, \| \cdot \| \right) \right\} \ = \   \\
      \sup \  \left\{ \ \frac{1}{4} \cdot  \sqrt{
           \left[ \left\| \vec{a} +  \vec{b} \right\|^{2} -   \left\| \vec{a} - \vec{b} \right\|^{2} \right]^{2} +
           \left[ \left\| \vec{a}  + {\bf i} \cdot \vec{b} \right\|^{2} - 
           \left\| \vec{a} - {\bf i} \cdot \vec{b} \right\|^{2} \right]^{2} }
            \ | \ \vec{a}, \vec{b} \in X, \ \| \vec{a} \| = 1 = \| \vec{b} \| \ . \right\} 
   \end{align*}    
   \end{definition} 
   \begin{proposition}
         Let  the pair \ $ (X, \| \cdot \|) $ \ be a complex normed vector space. \ % $ X \neq \{ \vec{0} \} $. \ 
         It holds     $$   1 \  \leq  \ { \cal D }(X, \| \cdot \|) \ \leq \ \sqrt{2}  \ .  $$
   \end{proposition} 
   \begin{proof} With \ $ \vec{b} := \vec{a} $ \ or \ $ \vec{b} :=  {\bf i} \cdot \vec{a} $ \ we get \ 
           $ 1 \  \leq  \ { \cal D }(X, \| \cdot \|) $. \ From Lemma \eqref{lemma eins} it follows \
           $  { \cal D }(X, \| \cdot \|) \ \leq \ \sqrt{2} $.
   \end{proof}  
   The above considerations suggest the following equivalence. One direction is trivial.
   \begin{conjecture}  In a complex normed vector space  \ $ (X, \| \cdot \|) $ \  the norm \ $ \| \cdot \| $ \ is 
        generated by an inner product by Definition \eqref{die erste definition} if and only if \
        $ { \cal D }(X, \| \cdot \|) \ = \ 1 $. 
   \end{conjecture}   
     $ $   \\  % \\   \\
   %   Equivalently, for $ \vec{x}, \vec{y} \neq \vec{0} $ it would mean the inequality  
   %   $$  \left| \cos(\angle(\vec{x}, \vec{y})) \right| \ = \ 
   %       \left| \ \frac{1}{4} \cdot  \left[ {\bf s}^{2}  - {\bf d}^{2} +
   %       {\bf i} \cdot  \left( {\bf s_{ i}}^{2} - {\bf d_{ i}}^{2}  \right)  \right] \ \right| \ \leq \ 1 \ . $$
   %   A further equivalence to the {\sf CSB }inequality is the following formulation: For all unit vectors \
   %       $ \vec{v}, \vec{w} $ \ in the complex normed vector space \ $ ( X, \| \cdot \| )$ \ it holds the inequality
   %   $$   \left[ \left\| \vec{v} +  \vec{w} \right\|^{2} - \left\| \vec{v} - \vec{w} \right\|^{2} \right]^{2} \ + \
   %        \left[ \left\| \vec{v} + {\bf i} \cdot \vec{w} \right\|^{2} - 
   %        \left\| \vec{v} - {\bf i} \cdot \vec{w} \right\|^{2} \right]^{2} \  \leq \ 16 \ . $$
   %           
          In the introduction we stated Theorem \eqref{allererstes theorem}. Here we catch up the proof.
     \begin{proof} 
          The property ({\tt An} 1) is a consequence of Lemma \eqref{lemma eins}.
          In Corollary \eqref{corollary always angle} it is said that the angle is always defined, and that the image 
          of the map $ \angle$ is situated in ${\cal A}$. 
                                 
          The five points ({\tt An} 2) - ({\tt An} 6) are rather trivial. 
          We use properties of the product  $ < \cdot \: | \cdot \: > $ from Proposition \eqref{proposition drei},
          and properties of the arccosine.   
          The next point ({\tt An} 7) is also easy.  
                                                                            
          We want to prove \ $ \angle(\vec{x}, \vec{y}) + \angle(-\vec{x}, \vec{y}) = \pi $, for 
          $ \vec{x}, \vec{y} \neq \vec{0} $. We use Proposition \eqref{proposition eins} and 
          $ < - \vec{x} \, | \, \vec{y} > = - < \vec{x} \, | \, \vec{y} > $ from 
          Proposition \eqref{proposition drei}.    
          If \ $ \angle(\vec{x}, \vec{y}) =  \arccos(r + {\bf i} \cdot s) = \frac{\pi}{2} + a + {\bf i} \cdot b $, \
          we have \
          $ \angle(-\vec{x}, \vec{y}) = \arccos(-r - {\bf i} \cdot s) = \frac{\pi}{2} - a - {\bf i} \cdot b $. \ 
          It follows  ({\tt An} 7).
                
          To prove the last point ({\tt An} 8) we use \cite{Thuerey2}.  We take the 
          two  linear independent vectors \ $ \vec{x},\vec{y}$ \ and we build the set $ \mathsf{U} $ of all 
          real linear combinations, $ \mathsf{U} := \{ r \cdot \vec{x} + s \cdot \vec{y} \ | \ r,s \in {\mathbb R}\} $.  
          The set  $\mathsf{U}$ is a real subspace of  $ X $ with the real dimension two. The norm in  $\mathsf{U}$ 
          is the induced norm of   $ ( X, \| \cdot \| )$, this makes the pair  $ ( \mathsf{U}, \| \cdot \| )$ 
          to a real subspace of   $ ( X, \| \cdot \| )$. \ Instead of $ \angle (\vec{x}, \vec{y} + t \cdot\vec{x})$
          we consider  the real part of the complex number \ 
          $ \cos (\angle (\vec{x}, \vec{y} + t \cdot\vec{x})) \in {\cal B}$. \ We define the map 
     $$    \widetilde{\Theta}(t) \ := \ % \frac{1}{4} \cdot  \left[ {\bf s}^{2}  - {\bf d}^{2} \right]  \ = \
                  \frac{1}{4} \cdot \left[ \left\| \frac{\vec{x}}{\|\vec{x}\|} 
                  +  \frac{\vec{y} + t \cdot\vec{x}}{\|\vec{y} + t \cdot\vec{x}\|} \right\|^{2} 
                  \ - \ \left\| \frac{\vec{x}}{\|\vec{x}\|} - \frac{\vec{y} + t \cdot\vec{x}}{\|\vec{y} + t
                  \cdot\vec{x}\|} \right\|^{2} \right] \, , \ \ \text{for} \ t \in  {\mathbb R} \ . $$
           By the triangle inequality we can regard this as a map \ 
           $ \widetilde{\Theta}:  {\mathbb R}\rightarrow [-1,1] $. The main theorem in \cite{Thuerey2} states
           that the map $\widetilde{\Theta}$ is an increasing homeomorphism onto the open
           interval $(-1,1 ) \subset  {\cal B} $. \ Since the complex arccosine function  is a homeomorphism 
           with domain $ {\cal B} $ and Codomain $ {\cal A} $, the first claim of ({\tt An} 8) is true. 
                             
           The limits are mentioned in the proof of the theorem in \cite{Thuerey2}, or we can find one 
           directly in \cite{Diminnie/Andalafte/Freese1}, which was the main source of \cite{Thuerey2}.
     \end{proof}           
 %%%%%%%%%%%%%%%%%%%%%%%%%%%%%%%%%%%%%%%%%%%%%%%%%%%%%%%%%%%%%%%%%%%%%%%%%%%%%%%%%%%%%%%%%%%%%%
     \begin{lemma}  \label{lemma zwei}
           In a complex normed vector space  $ ( X, \| \cdot \| )$ we take two vectors 
           $ \vec{x}, \vec{y} \neq \vec{0}$.   It holds    
           $ |\cos(\angle(\vec{x}, \vec{y}))| = |\cos(\angle(-\vec{x}, \vec{y}))|  = 
            |\cos(\angle(  {\bf i} \cdot \vec{x}, \vec{y}))|  =  |\cos(\angle(\vec{x},  {\bf i} \cdot \vec{y}))| $. 
    \end{lemma}  	 
    \begin{proof}
          This fact follows easily with Proposition \eqref{proposition drei}.
       %   and  Definition \eqref{die zweite definition}. 
    \end{proof}
   %   In the following table we note some angles and their cosines.  
      We assume two elements $ \vec{x},\vec{y} \neq \vec{0} $ of a complex normed space $ ( X, \| \cdot \| )$,
      and we repeat the table from the introduction.    After the table some explanations were given.  \\
   %  We take six suitable real numbers \ $ a, b, r, s, v, w $, note 
   %   $ - \frac{\pi}{2} \leq a, v \leq \frac{\pi}{2}$ \ and \ $ -1 \leq r \leq 1 $, \ and we assume three complex
   %   numbers \\ 
   %   \centerline {   $ \angle (\vec{x},\vec{y}) =: \frac{\pi}{2} + a + {\bf i} \cdot b \in {\cal A}$,   } \\  
   %    \centerline {  $ \cos(\angle (\vec{x},\vec{y})) = 
   %    \cos(\frac{\pi}{2} + a + {\bf i} \cdot b) =: r + {\bf i} \cdot s \in {\cal B}$, \ \ and    }  \\  
   %    \centerline {  $ \angle ({\bf i} \cdot \vec{x},\vec{y}) =: 
   %                            \frac{\pi}{2} + v + {\bf i} \cdot w \in {\cal A}$. } \\ 
   \begin{center}  
   \begin{tabular}{c|l|l|c|c}     \hline
        pair of vectors  & their angle $ \angle $ & the cosine of $ \angle $ & the angle for $ \vec{x} = \vec{y} $ 
                                                                         & its cosine for $ \vec{x} = \vec{y} $  \\ 
        \hline\hline
        $(\vec{x},\vec{y})$  & $  \frac{\pi}{2} + a + {\bf i} \cdot b $     
                             &  $ r + {\bf i} \cdot s $ & $ 0 $ & $ 1 $  \\
        $(-\vec{x},\vec{y})$  & $  \frac{\pi}{2} - a - {\bf i} \cdot b $     
                              & $ -r - {\bf i} \cdot s $ & $ \pi $ & $ -1 $   \\ 
        $(\vec{y},\vec{x})$  & $  \frac{\pi}{2} + a - {\bf i} \cdot b $  
                           & $ r - {\bf i} \cdot s $ & $ 0 $ & $ 1 $     \\
        $(-\vec{y},\vec{x})$  & $  \frac{\pi}{2} - a + {\bf i} \cdot b $    
                              & $ -r + {\bf i} \cdot s $   & $ \pi $ & $ -1 $   \\    \hline
        $({\bf i} \cdot \vec{x},\vec{y})$  & $ 
                   \frac{\pi}{2} + v + {\bf i} \cdot w  $ &  $ -s + {\bf i} \cdot r $ & 
                    $\frac{\pi}{2} - {\bf i} \cdot \log \left[ \sqrt{2}+1 \right]$  &  ${\bf i}$  \\
        $(\vec{y},{\bf i} \cdot \vec{x})$  & $  
                   \frac{\pi}{2} + v - {\bf i} \cdot w  $ &  $ -s - {\bf i} \cdot r $ &
                    $\frac{\pi}{2} + {\bf i} \cdot \log \left[ \sqrt{2}+1 \right]$  &  $-{\bf i}$   \\ 
        $(\vec{x},{\bf i} \cdot \vec{y})$  & $    
                   \frac{\pi}{2} - v - {\bf i} \cdot w $ &  $  s - {\bf i} \cdot r $ &
                      $\frac{\pi}{2} + {\bf i} \cdot \log \left[ \sqrt{2}+1 \right]$  &  $-{\bf i}$   \\   
        $({\bf i} \cdot \vec{y},\vec{x})$  & $   
                   \frac{\pi}{2} - v + {\bf i} \cdot w $ &  $  s + {\bf i} \cdot r $ &
                      $\frac{\pi}{2} - {\bf i} \cdot \log \left[ \sqrt{2}+1 \right]$  &  ${\bf i}$   \\  \hline
    %    $( -{\bf i} \cdot \vec{y},\vec{x})$  & $    
    %               \frac{\pi}{2} + v - {\bf i} \cdot w $ &  $ -s - {\bf i} \cdot r $ &
    %                  $\frac{\pi}{2} + {\bf i} \cdot \log \left[ \sqrt{2}+1 \right]$  &  $-{\bf i}$   \\   \hline 
   \end{tabular} 
   \end{center}     $ { } $    \\ \\
      Note that in the table the values \ $ \angle (\vec{x},\vec{y}) = \frac{\pi}{2} + a + {\bf i} \cdot b $ \ and \ 
      $ \cos( \angle (\vec{x},\vec{y})) =  r + {\bf i} \cdot s $ \ are by definition, as well the expression 
      $ \angle ({\bf i} \cdot \vec{x},\vec{y}) = \frac{\pi}{2} + v + {\bf i} \cdot w $.     
   \begin{proof}
      The final column comes directly from Definition \eqref{die zweite definition}, e.g.  
      $\cos(\angle (\vec{x},\vec{x})) = < \vec{x} | \vec{x} > / \|\vec{x}\|^{2} $, etc.  
      We get the values of the fourth column by using the fifth column and applying the arccosine function
      from Definition \eqref{die dritte definition}. The other columns have to be discussed. 
                                                     
      We show, for instance, the final line \   
      $ \angle ({\bf i} \cdot \vec{y}, \vec{x}) = \frac{\pi}{2} - v + {\bf i} \cdot w$, \ and \
      $ \cos(\angle ({\bf i} \cdot \vec{y},\vec{x}))  =    s + {\bf i} \cdot r $.  \
  %      We had defined  $ \angle (\vec{x},\vec{y}) = \frac{\pi}{2} + a + {\bf i} \cdot b $ \ and \ 
  %      $\cos(\angle (\vec{x},\vec{y})) = r + {\bf i} \cdot s $. \ 
      We compute  
      $$ \angle ({\bf i} \cdot \vec{y},\vec{x})  = 
         \arccos{ \left( \frac{< {\bf i} \cdot \vec{y} \: | \: \vec{x} > }
         { \|{\bf i} \cdot \vec{y}\| \cdot \|\vec{x}\|} \right)}  =   
         \arccos{ \left({\bf i} \cdot \frac{< \vec{y} \: | \: \vec{x} > } {\|\vec{y}\| \cdot \|\vec{x}\| }\right)}     
          =  \arccos{ \left({\bf i} \cdot \frac{ \overline{< \vec{x} \: | \: \vec{y} >} }
            {\|\vec{y}\| \cdot \|\vec{x}\| }\right)} \ , $$
       hence it follows 
     \begin{align}  \label{gleichung cosinusimalx}
           \cos(\angle ({\bf i} \cdot \vec{y},\vec{x})) \ = \ 
     %      \cos \left( \arccos{ \left({\bf i} \cdot \frac{ \overline{< \vec{x} \: | \: \vec{y} >} }
     %      {\|\vec{y}\| \cdot \|\vec{x}\| }\right)} \right)  =
           {\bf i} \cdot \frac{ \overline{< \vec{x} \: | \: \vec{y} >} } {\|\vec{y}\| \cdot \|\vec{x}\| }
           \ = \ {\bf i} \cdot \overline{ \left(r + {\bf i} \cdot s \right)} \ = \ s +  {\bf i} \cdot r   \ .   
     \end{align}        
       With a similar argumentation we get  
    $$ \cos(\angle ({\bf i} \cdot \vec{x},\vec{y}))  = -s + {\bf i} \cdot r \ .          $$
        We had defined \ $ \angle ({\bf i} \cdot \vec{x},\vec{y}) = 
                               \frac{\pi}{2} + v + {\bf i} \cdot w $,  hence it follows \ 
          $  \frac{\pi}{2} + v + {\bf i} \cdot w \ = \ \arccos{( -s + {\bf i} \cdot r )} $.  \\
          In Definition  \eqref{die dritte definition} we introduced the  arccosine function. 
      %    We use  Proposition \eqref{proposition eins}, and   ??????????????????
          By the different  signs of \ $ s $ \ we can deduce \ 
          $ \arccos{( s + {\bf i} \cdot r )} = \frac{\pi}{2} - v + {\bf i} \cdot w $.  \ 
          With equation   \eqref{gleichung cosinusimalx} we get  
          $ \angle ({\bf i} \cdot \vec{y},\vec{x})  =  \frac{\pi}{2} - v + {\bf i} \cdot w $, 
          and the final line of the table is shown. \\
          The other lines can be proven with similar considerations.  
     \end{proof} 
        With the help of the table we can conclude other values, for instance we get 
        $ \angle( -{\bf i} \cdot \vec{y},\vec{x}) = \frac{\pi}{2} + v - {\bf i} \cdot w = 
        \angle( \vec{y}, {\bf i} \cdot\vec{x}) $ \ with its cosine \ 
        $ \cos(\angle( -{\bf i} \cdot \vec{y},\vec{x})) = -s - {\bf i} \cdot r $.  
                                
        We refer to the above table, where we had assumed \ 
        $ \angle (\vec{x},\vec{y}) = \frac{\pi}{2} + a + {\bf i} \cdot b $. 
       If we want to express \ $ \cos(\angle (\vec{x},\vec{y})) =r + {\bf i} \cdot s $ \ in coordinates of
       $ a $ and $ b $,  we get at once from Definition \eqref{die dritte definition} 
     \begin{align}  \label{gleichung fuer cosinus winkelxy}
          \  r + {\bf i} \cdot s \ = \ \cos\left(\frac{\pi}{2} + a + {\bf i} \cdot b \right)  
          \ = \ \frac{1}{2} \cdot 
          \left[ \cos\left(\frac{\pi}{2} + a \right) \cdot \left(e^{b} + \frac{1}{e^{b}} \right) - {\bf i} \cdot
          \sin\left(\frac{\pi}{2} + a\right) \cdot \left(e^{b} - \frac{1}{e^{b}} \right) \right] \ .    
      \end{align}    
     %   \\  \\  
        To express the complex number \ 
        $  \angle ({\bf i} \cdot \vec{x},\vec{y}) = \frac{\pi}{2} + v + {\bf i} \cdot w $ \
        in dependence of $ a $ and $ b $ we have to make some more effort. In the introduction in 
        Theorem \eqref{theorem winkelixy} we already have presented the result. We need to prove it. 
   %    We used the real valued $ \cos $ and $ \cosh $ functions and their inverses  $ \arccos $ and arcosh.   
    \begin{proof}
         Here we also use the real sine and hyperbolic sine functions, abbreviated by  $ \sin $ and $ \sinh $,  
    %    \ $ \sinh(x) :=  \frac{1}{2} \cdot  \left(e^{x} - \frac{1}{e^{x}} \right) \ \text{for} \ x \in {\mathbb R}$.  
         please see Definition  \eqref{die dritte definition}.
         To shorten the presentation of the proof it is useful to introduce more abbreviations. \\
         \centerline{ Let \ $  {\sf cpi2a} :=  \cos\left(\frac{\pi}{2} + a\right), \    
                                              {\sf spi2a} :=  \sin\left(\frac{\pi}{2} + a\right)$ . }  \\
     %   {\sf coshb} :=   \cosh(b), \   \ {\sf sinhb} :=   \sinh(b) $.            
         From equation  \eqref{gleichung fuer cosinus winkelxy} we have \    
    $$   r + {\bf i} \cdot s \ = \ \cos(\angle(\vec{x}, \vec{y})) 
      %   \ = \ \cos\left(\frac{\pi}{2} + a + {\bf i} \cdot b \right) 
         \ = \ {\sf cpi2a} \cdot  \cosh(b) - {\bf i} \cdot  {\sf spi2a} \cdot \sinh(b) , $$ 
         and we have two  real numbers \ $ r =  {\sf cpi2a} \cdot  \cosh(b) $ \ and \
         $ s =  - {\sf spi2a} \cdot \sinh(b) $. 
         Note \ $ - \frac{\pi}{2} \leq a \leq \frac{\pi}{2}$. Since \ 
         $ \frac{\pi}{2} + a + {\bf i} \cdot b \in {\cal A} $ \ it follows from the special cases \
         $ a = - \frac{\pi}{2} $ \ or \ $ a = \frac{\pi}{2} $ \ that the imaginary part $ b $ vanished,
         i.e. $ 0 = b = s = {\sf spi2a} =  \sinh(b)$. 
                                                                             
         We get from the above table and Definition \eqref{die dritte definition} \ 
     \begin{align}  
           &  \angle({\bf i} \cdot \vec{x},\vec{y}) \ = \ \frac{\pi}{2} + v + {\bf i} \cdot w \ = \
              \arccos( -s + {\bf i} \cdot r ) \\ 
           &  =  \ \arccos   [ ( {\sf spi2a} \cdot \sinh(b) ) + {\bf i} \cdot ( {\sf cpi2a} \cdot  \cosh(b)) ] \\
           &  = \ \frac{\pi}{2}  - 
                  \arcsin [( {\sf spi2a} \cdot \sinh(b) ) + {\bf i} \cdot ( {\sf cpi2a} \cdot \cosh(b) )] \\
           &  = \ \frac{\pi}{2} - \frac{1}{2} \cdot   \left[ {\rm sgn}({\sf spi2a} \cdot \sinh(b) ) \cdot
                  \arccos({\mathsf{K}}_-) + {\bf i} \cdot {\rm sgn}( {\sf cpi2a} \cdot \cosh(b) ) \cdot 
                  {\rm arcosh}({\mathsf{K}}_+ ) \right]  \\ 
           &  = \ \frac{\pi}{2} - \frac{1}{2} \cdot   \left[ {\rm sgn}(b) \cdot
                  \arccos({\mathsf{K}}_-) + {\bf i} \cdot {\rm sgn}( -a ) \cdot 
                  {\rm arcosh}({\mathsf{K}}_+ ) \right] \, ,  
        \end{align}                    
           with the abbreviations \ ${\mathsf{K}}_{-}$ and  ${\mathsf{K}}_{+}$, 
        \begin{align*}   
           & {\mathsf{K}}_{\pm} :=
                       \sqrt{\left[ {\sf spi2a}^{2} \cdot \sinh^{2}(b) + {\sf cpi2a}^{2} \cdot \cosh^{2}(b)
                             -1 \right]^{2} + 4 \cdot {\sf cpi2a}^{2} \cdot \cosh^{2}(b) } \\
           & \qquad  \quad   \qquad \qquad  \qquad \qquad \quad  \qquad \qquad \qquad \qquad \quad  \qquad
            \pm \left[ {\sf spi2a}^{2} \cdot \sinh^{2}(b) + {\sf cpi2a}^{2} \cdot \cosh^{2}(b) \right] \ . 
       \end{align*}          
           With the aid of the well-known equations 
       $$     \sin^{2}(x) + \cos^{2}(x) \ = \ 1 \ = \  \cosh^{2}(x) - \sinh^{2}(x)      $$ 
           we finally reach the identities \  
                 $ {\mathsf{H}}_- = {\mathsf{K}}_- \ \text{and} \ {\mathsf{H}}_+ = {\mathsf{K}}_+ $, \ 
           which was the last step to prove Theorem \eqref{theorem winkelixy}.      
    \end{proof}
    Now we prove Corollary \eqref{corollary pure real angle}.
  \begin{proof}
      Since \ $ b = 0 $ \ we have \ $ \sinh(b) = 0 $ \ and $ s = 0 $, and \  $ \cosh(b) = 1 $.  It follows 
        \begin{align*}
           &  \angle({\bf i} \cdot \vec{x},\vec{y}) \ = \ \frac{\pi}{2} + v + {\bf i} \cdot w \ = \
              \arccos( {\bf i} \cdot r )    % \\ 
            \ =  \ \arccos ( {\bf i} \cdot {\sf cpi2a} \cdot \cosh(b) ) \\  
           &  = \ \frac{\pi}{2} - \frac{1}{2} \cdot   \left[ {\bf i} \cdot {\rm sgn}( -a ) \cdot 
                                            {\rm arcosh}({\mathsf{K}}_+ ) \right]   \\ 
           &  = \ \frac{\pi}{2} + \frac{1}{2} \cdot  {\bf i} \cdot {\rm sgn}( a ) \cdot 
                {\rm arcosh} \left[ 2 \cdot \cos^{2}\left(\frac{\pi}{2} + a \right) + 1 \right] .     
        \end{align*} 
       Please see the definition of the arcosh function in Definition \eqref{die dritte definition}, 
       and note  that the equation
    $$  \log \left[ \ \sqrt{\cos^{2}(x) + 1} + |\cos (x)| \ \right] \ = \  \frac{1}{2} \cdot \log \left[ \ 2 \cdot 
             \cos^{2}(x)+ 1 + 2 \cdot |\cos(x)| \cdot \sqrt{\cos^{2}(x)+ 1} \ \right]  $$   
       holds for all real numbers $x$, which concludes the proof.    % \\  \\
  %  5.2.2013: \ Es hakt noch mitm Vorzeichen  vom Cosinus  !!!!!!  Kontrolle  !!!!!!!!!!! \ \
  %  Oben geht Vorzeichen von a einmal ein, darunter aber 2 mal !!!!!.  Da is' was faul. \\
  %  Verbessert am 7.2.2013 !!!!!!!!!!!!!!!!!!!!!!!!!!!  Nochmal kontrollieren !!!!!!!!!!!!!!!!!!  
  \end{proof} 
   We add the proof of Corollary \eqref{corollary real part pidurchzwei angle}.
   \begin{proof}
     First a lemma. 
     \begin{lemma}  \label{range of b}
       In the case of \ $ \angle (\vec{x},\vec{y}) = \frac{\pi}{2} + {\bf i} \cdot b $, i.e. $ a = 0 $, 
       the range of \ $ b $ \ is 
       $$ -\log \left( \sqrt{2} + 1 \right) \ \leq \ b \ \leq \ +\log\left( \sqrt{2} + 1\right) \ \approx \ 0.88 . $$ 
     \end{lemma} 
     \begin{proof} By Lemma \eqref{lemma eins}, there is a suitable \ $ -1 \leq s \leq +1 $ \ with \
      $ \cos\left(\frac{\pi}{2} + {\bf i} \cdot b\right) =  {\bf i} \cdot s $. \ By using the arccosine
      of Definition \eqref{die dritte definition}  it follows for the modulus of $ b $   
     $$ |b|  \ = \ \frac{1}{2} \cdot {\rm arcosh} ({\mathsf{G}}_+) \ = \
                 \frac{1}{2} \cdot {\rm arcosh} \left( 2 \cdot s^{2} + 1 \right) \ = \ 
                 \log \left( \sqrt{ 2 \cdot s^{2} + 1 + 2 \cdot |s| \cdot \sqrt{s^{2} + 1} } \right) \ , $$ 
                 and note \  $ \sqrt{3 + 2 \cdot \sqrt{2}} \ = \ \sqrt{2} + 1 $, \ and the lemma is proven. 
     \end{proof}
         We apply Theorem \eqref{theorem winkelixy}, and since \ $ a = 0 $ \ we get 
         \begin{align*}
          \angle ({\bf i} \cdot \vec{x},\vec{y}) \ = \ %   \frac{\pi}{2} + v + {\bf i} \cdot w  
            \frac{\pi}{2} + \frac{1}{2} \cdot \left[ - {\rm sgn}(b) \cdot \arccos({\mathsf{H}}_-) \right] \ , \ 
           \text{with} \ \ 
           {\mathsf{H}}_{-} = \sqrt{\left[ \cosh^{2}(b) - 2 \right]^{2}} - \left[ \cosh^{2}(b) - 1 \right] \ . 
         \end{align*}   
         A consequence of Lemma \eqref{range of b} is the fact \ $ \cosh^{2}(b) \leq 2 $, \ it follows 
         $$ {\mathsf{H}}_{-} \ = \ \left[ 2 - \cosh^{2}(b) \right] - \left[ \cosh^{2}(b) - 1 \right] 
                             \ = \ 3 - 2 \cdot \cosh^{2}(b) \ , $$ 
         and the first line of Corollary \eqref{corollary real part pidurchzwei angle} is proven. 
         The next line is some calculus.  
 %         Nochmal checken !!!!!!!!!!!!!!!!!!!!!!!!!!!!!!!!!!!!!!!!!!!!!!!!!  
   \end{proof}   
 %  We notice a few special angles for an element $ \vec{x} \neq \vec{0}  $. It holds  
 %  $ \angle (\vec{x},\vec{x}) = 0, \ \angle (-\vec{x},\vec{x}) = \pi$, \ 
 %  $\angle ({\bf i} \cdot \vec{x},\vec{x}) = \frac{\pi}{2} - {\bf i} \cdot \log \left[ \sqrt{2}+1 \right], \
 %  \angle ( \vec{x}, {\bf i} \cdot \vec{x}) = \frac{\pi}{2} + {\bf i} \cdot \log \left[ \sqrt{2}+1 \right] 
 %   =  \angle ( - {\bf i} \cdot \vec{x}, \vec{x}) $,  \  where the first two  values come directly from the
 %   definitions  \eqref{die erste definition} and \eqref{die zweite definition}, the others
 %   from Corollary \eqref{corollary pure real angle} or from the table. 
 %  (musß stimmen wegen (An 7) !!!!   \\  \\
 %  oder vielleicht $ \angle ( - {\bf i} \cdot \vec{x}, \vec{x}) = \frac{\pi}{2} + {\bf i} \cdot \log[\sqrt{2}-1]$, \\ 
 %   ( nach  Karola  Corollary  \eqref{corollary pure real angle}) noch überhaupt  garnicht klar !!!!!!!!  \\  \\
 %                                
   We notice a few interesting facts about the general product $ < . \, | \, . > $ 
   from Definition \eqref{die erste definition}.                                
   \begin{lemma} 
      In a complex normed space  $ (X, \| \cdot \|) $ for $ \vec{x} \in X$ and real $ \varphi $ 
      there is the identity   \\
      \centerline {  $< e^{{\bf i} \cdot \varphi} \cdot \vec{x} \, | \, \vec{x} > \ = \
                     e^{{\bf i} \, \cdot \varphi} \cdot   < \vec{ x} \, | \, \vec{x} > $  \  . }
  \end{lemma} 
   \begin{proof}   Straightforward. Write \
   $ e^{{\bf i} \, \cdot \varphi} = \cos(\varphi) + {\bf i} \, \cdot \sin(\varphi) $, and use Definition 
   \eqref{die erste definition}  
 %   Nochmal checken !!!!!!!!!!!!!!!!!!!!!!!!!!!!!!!!!!!!!!!!!!!!!!!!!  
   \end{proof}  
   \begin{corollary} For an unit vector $ \vec{x} \in (X, \| \cdot \|) $ we have that the set
       $ \{ < e^{{\bf i} \cdot \varphi} \cdot \vec{x} \, | \, \vec{x} > \ | \  \varphi \in [ 0, 2 \, \pi ] \} $
       is the complex unit circle, since \ 
       $ < e^{{\bf i} \cdot \varphi} \cdot \vec{x} \, | \, \vec{x} > \ = \
       e^{{\bf i} \cdot \varphi} \cdot  < \vec{x} \, | \, \vec{x} > \ = \ e^{{\bf i} \cdot \varphi} $. 
   \end{corollary}
     \begin{remark}   \rm   The next example shows that in a complex normed space  $ (X, \| \cdot \|) $ 
        generally we have the inequality \ $ < e^{{\bf i} \cdot \varphi} \cdot \vec{x} \, | \, \vec{y} > \ \neq \
                                           e^{{\bf i} \, \cdot \varphi} \cdot   < \vec{x} \, | \, \vec{y} > $. \
        This means that in this case the set of products \
        $ \left\{ < e^{{\bf i} \cdot \varphi} \cdot \vec{x} \, | \, \vec{y} > \ | \ \varphi \in [ 0, 2 \pi ] \right\}$
        \ does not generate a proper Euclidean circle   
        $ ( \text{with radius} \ |< \vec{x} \, | \, \vec{y} >| )$ in $ {\mathbb C} $.  But with    %  $ {{\cal B}} $. 
        Proposition \eqref{proposition drei}  we can be sure that we have three identities  \\
        \centerline{$ < - \vec{x} \, | \, \vec{y} > \ = \ - < \vec{x} \, | \, \vec{y} > \ , \ 
        < {\bf i} \cdot \vec{x} \, | \, \vec{y} > \ = \  {\bf i} \, \cdot  < \vec{x} \, | \, \vec{y} > \ , \  \text{and}
        \ < -{\bf i} \cdot \vec{x} \, | \, \vec{y} > \ = \ -{\bf i} \, \cdot  < \vec{x} \, | \, \vec{y} > $.}             \end{remark}   
  \begin{lemma} 
    In a complex normed space  $ (X, \| \cdot \|) $ generally it holds the inequality   \\
      \centerline {  $ < e^{{\bf i} \cdot \varphi} \cdot \vec{x} \, | \, \vec{y} > \ \neq \
                     e^{{\bf i} \, \cdot \varphi} \cdot   < \vec{x} \, | \, \vec{y} > $, \
                       even their moduli are different.       }
   \end{lemma} 
   \begin{proof}   
      The lemma can be deduced by Theorem \eqref{conjecture csb}, but we make a direct proof.              
      We use the most simple non-trivial example of a complex
      normed space, let \ $ (X, \| \cdot \|) := ( {\mathbb C} \times {\mathbb C} , \| \cdot \|_{\infty} ) $, where
      for two complex numbers \ $ r + {\bf i} \cdot s \, , \  v + {\bf i} \cdot w \in  {\mathbb C} $ \ 
      we get its norm ß $ \| \cdot \|_{\infty} $ \ by 
      $$  \left\| \left(  \begin{array}{c} r + {\bf i} \cdot s  \\ v + {\bf i} \cdot w  \end{array}  \right)  
                  \right\|_{\infty} \ = \ \max \left\{ \sqrt{r^{2}+s^{2}} , \sqrt{v^{2}+w^{2}} \right\} \ .   $$ 
      We define two unit vectors \ $ \vec{x}, \vec{y} $ \ of \ 
      $( {\mathbb C} \times {\mathbb C} , \| \cdot \|_{\infty} ) $, \ let 
      $$   \vec{x} \ := \ \frac{1}{4} \cdot 
          \left( \begin{array}{cr} 1 + {\bf i} \cdot \sqrt{15} \\ 2 + {\bf i} \cdot 2 \end{array}  \right)                         \ \quad \text{and} \ \ \   
           \vec{y} \ := \ \frac{1}{4} \cdot 
          \left( \begin{array}{c} 2 + {\bf i} \\ 3 + {\bf i} \cdot \sqrt{7} \end{array}  \right) \ .       $$
    %    with elements $  \vec{x} , \vec{y} $ and a real number $ \varphi \in [0, 2 \pi ] $
       Some calculations yield the complex number 
    $$ < \vec{x} \, | \, \vec{y} > \ = \ 
             \frac{1}{64} \cdot \left( \ 19 + 4 \cdot \sqrt{7} + 2 \cdot \sqrt{15}
             + {\bf i} \cdot \left[ 7 - 4 \cdot \sqrt{7} + 4 \cdot \sqrt{15} \right] \ \right) \ 
             \approx \ 0.583 + {\bf i} \cdot 0.186  \ . $$           
    We choose \ $ e^{{\bf i} \cdot \varphi} := 1/2 \cdot \left( 1 +  {\bf i} \cdot \sqrt{3}\right) $ \ 
        from the complex unit circle, and we get approximately \ 
        $  e^{{\bf i} \cdot \varphi} \cdot < \vec{x} \, | \, \vec{y} > \ \approx \ 0.130 + {\bf i} \cdot 0.598 . $  
        \ After that we take the unit vector 
    $$   e^{{\bf i} \cdot \varphi} \cdot \vec{x} \ = \ \frac{1}{8} \cdot 
          \left( \begin{array}{c} 1 - \sqrt{45} + {\bf i} \cdot \left[ \sqrt{3} + \sqrt{15} \right]
          \\ 2 - 2 \cdot \sqrt{3} + {\bf i} \cdot \left[ 2 + 2 \cdot  \sqrt{3} \right] \end{array}  \right) \ ,  $$  
      and we compute the product \  $ < e^{{\bf i} \cdot \varphi} \cdot \vec{x} \, | \, \vec{y} > \ $. 
     We get the result 
   \begin{align*}      
      < e^{{\bf i} \cdot \varphi} \cdot \vec{x} \, | \, \vec{y} > \ & = \ \frac{1}{64} \cdot 
               \left( p \ + {\bf i} \cdot q \right) 
        %     \frac{ p \ + {\bf i} \cdot q }{64}  \ 
             \approx \ 0.113 + {\bf i} \cdot 0.628 , \ \text{where}  \ p  \ \text{and}  \ q  \ 
             \text{abbreviate the real numbers}  \\    
          p \ & = \ 11 + 2 \cdot \left( \sqrt{7} + \sqrt{21} - \sqrt{45} \right) - 5 \cdot \sqrt{3} + \sqrt{15} \ , \\ 
          q \ & = \ 8 + 2 \cdot \left( 4 \cdot \sqrt{3} - \sqrt{7} + \sqrt{15} + \sqrt{21} \right) + \sqrt{45} \ . 
   \end{align*}               
       This  proves the inequality \
        $ < e^{{\bf i} \cdot \varphi} \cdot \vec{x} \, | \, \vec{y} > \ \neq \
          e^{{\bf i} \, \cdot \varphi} \cdot   < \vec{x} \, | \, \vec{y} > $, \ and the lemma is confirmed.    
   \end{proof} 
    The above lemma suggests the following conjecture. One direction is trivial.
    \begin{conjecture}
          In a complex normed space  $ (X, \| \cdot \|) $  for all \ $ \vec{x}, \vec{y} \in X $ it holds the equality \\
      \centerline {  $ < e^{{\bf i} \cdot \varphi} \cdot \vec{x} \, | \, \vec{y} > \ = \
                     e^{{\bf i} \, \cdot \varphi} \cdot   < \vec{x} \, | \, \vec{y} > $  }  \\
                     if and only if its product  $ < \cdot \; | \; \cdot >  $ from
          Definition \eqref{die erste definition} is actually an inner product, i.e. 
          $ (X,  < \cdot \; | \; \cdot >) $ is an inner product space.  
    \end{conjecture}           
        %            $ { }  $ \\ \\
    Up to now we had defined for each complex normed space $ X $ an `angle' which generally has complex values. 
    The geometrical meaning of a complex angle is unclear. But to do the usual known `Euclidean' geometry 
    we need real valued angles.  
    During the following consideration it turns out that although we deal with complex vector spaces
    actually `a lot' of our angles are pure real.  The situation will even improve in inner product spaces, 
    which will be investigated in the next section. First we  create a name for these `good' pairs with a
    real product, which include all pairs with a real angle.  
    \begin{definition}  \label{paare mit reellem winkel}
    For a complex vector space $ X $ provided with a norm  $ \| \cdot \|$ we define the set of pairs \
     $ {\cal R}_X \subset X \times X $ \ with real products by \ 
     $ {\cal R}_X := \{ (\vec{x}, \vec{y}) \ | \ \vec{x}, \vec{y} \in X \
     \text{and} \  < \vec{x} \, | \, \vec{y} > \, \in {\mathbb R} \} $.  
                         
    Further we denote by \ $ {\cal R}_X^{\bf \bullet} $ those pairs $(\vec{x}, \vec{y}) \in 
     {\cal R}_X $ with a pure real angle, i.e. \ $\angle(\vec{x}, \vec{y}) \in {\mathbb R} $.
   %   both  \ $\vec{x}\neq \vec{0}$ \ and \ $\vec{y} \neq \vec{0}$.  
   %   \quad $ {\cal R}_X^{\bf \bullet } $ 
   %       \\   \\  $ \oslash  \ominus   \sqsupseteq\succeq\asymp\doteq\doteq\models\perp\bowtie\smile\frown $
    \end{definition}                                                   
    Note that the `diagonal' \ $ \{ (\vec{x},\vec{x}) \ | \ \vec{x} \in X \} $ is a subset of ${\cal R}_X $. \\
    The  following proposition  shows that the set $ {\cal R}_X^{\bf \bullet} $ \ of pairs with real angles 
    is \, `rather large'.  
  \begin{proposition} \label{proposition vier} 
         Let us take two vectors $ \vec{x}, \vec{y} \neq  \vec{0} $. It holds 
         $$  \{ ( e^{{\bf i} \cdot \varphi} \cdot \vec{x} \, , \, \vec{y} ) \ | \  \varphi \in [ 0, 2 \pi ] \}  
                                        \ \cap \  {\cal R}_X^{\bf \bullet}  \ \neq \ \emptyset \ .         $$ 
  \end{proposition} 
         Of course, the parts of $ \vec{x}$ and $ \vec{y} $ can be exchanged. 
         The proposition means that for $ \vec{x}, \vec{y} \neq \vec{0} $ we have to `twist' either  $ \vec{x} $ or 
         $ \vec{y} $ by a suitable complex factor  $  e^{{\bf i} \cdot \varphi } $ to generate a pure real angle.    
  \begin{proof}
        Please see both Proposition \eqref{proposition eins} and Proposition \eqref{proposition drei}. 
        Let us assume a complex angle   \\
        $$ \angle(\vec{x}, \vec{y})  \ = \ 
        \arccos{ \left( \frac{< \vec{x} \, | \, \vec{y} > }{ \|\vec{x}\| \cdot \|\vec{y}\| } \right) } 
                    \ = \ \frac{\pi}{2} + a + {\bf i} \cdot b \ \in {\cal A} , \ \ \text{with} \ b \neq 0 \ . $$ 
        From  Proposition \eqref{proposition drei} we have  
        $ < - \vec{x} \, | \, \vec{y} > \ = \ - <  \vec{x} \, | \, \vec{y} > $, \ i.e.
        with Proposition \eqref{proposition eins} it follows
        $$  \angle( -\vec{x}, \vec{y})  \ = \ 
        \arccos{ \left( \frac{< -\vec{x} \, | \, \vec{y} > }{ \|\vec{x}\| \cdot \|\vec{y}\| } \right) }  
    %   \ = \  \arccos{ \left( - \; \frac{< \vec{x} \, | \, \vec{y} > }{ \|\vec{x}\| \cdot \|\vec{y}\| } \right) } 
                    \ = \ \frac{\pi}{2} - a - {\bf i} \cdot b \ . $$  
                  %    \left( 
        We know $ e^{{\bf i} \cdot \pi } = -1 $. The set \ 
  %      $ Oval(\vec{x},\vec{y}) := 
  %      \left\{ \arccos{ \left( \frac{ < e^{{\bf i} \cdot \varphi} \cdot \vec{x} \, | \, \vec{y} > }
  %      { \|\vec{x}\| \cdot \|\vec{y}\| } \right) } \ | \  \varphi \in [ 0, 2 \pi ] \right\} \subset {\cal A} $ \\  
         $ Oval(\vec{x},\vec{y}) := \left\{ \angle( e^{{\bf i} \cdot \varphi} \cdot \vec{x}, \vec{y})  
                                \ | \  \varphi \in [ 0, 2 \pi ] \right\} \subset {\cal A} $ \ 
        is the continuous   image  of the connected complex unit circle 
        \ $ \left\{ e^{{\bf i} \cdot \varphi}  \ | \  \varphi \in [ 0, 2 \pi ] \right\} $, 
        or the interval $ [ 0, 2 \pi ] $, respectively, therefore it has to be connected. This means that \ 
        $ Oval(\vec{x},\vec{y}) $ \ is connected, i.e. it must cross the real  axis.       
  \end{proof} 
              
        Let's turn to inner product spaces.                          
 %%%%%%%%%%%%%%%%%%%%%%%%%%%%%%%%%%%%%%%%%%%%%%%%%%%%%%%%%%%%%%%%%%%%%%%%%%%%%%%%%%%%%%%%%%%%%%%%%%%% 
  %  \newpage 
  \section{Complex Inner Product Spaces}    \label{section five} 
          
     In the introduction we construct in Definition \eqref{die erste definition} a continous product 
     for all complex normed spaces  $ (X, \| \cdot \|) $. There we already mentioned that in a case of 
     an inner product space 
     $ (X, < \cdot \: | \: \cdot >) $ the product from  Definition \eqref{die erste definition} coincides with 
     the given inner product  $ < \cdot \: | \: \cdot > $, which can be described as in 
     equation \eqref{allererste definition}. 
     For all complex normed spaces $ (X, \| \cdot \|) $ we introduced an `angle' $ \angle $ \
     in Definition \eqref{die zweite definition}. \ Now we investigate its 
     properties in the special case of an inner product space  $ (X, < \cdot \: | \: \cdot >) $.
  %%%%%%%%%%%%%%%%%%%%%%%%%%%%%%%%%%%%%%%%%%%%%%%%%%%%%%%%%%%%%%%%%%%%%%%%%%%%%%%%%%%%%%%%%%%%%%%%%%%%%%%%%%%
  %%  \newpage
                                       
     We deal with complex vector spaces $ X $ provided with an inner product  
     $  < \cdot \: | \: \cdot > $, i.e. it has the properties 
     $\overline{(1)}, \overline{(2)}, \overline{(3)}, \overline{(4)} $. Further, the parallelogram identity \
     $\|\vec{x}+\vec{y}\|^{2} + \|\vec{x}-\vec{y}\|^{2} = 2 \cdot  \left(\|\vec{x}\|^{2} + \|\vec{y}\|^{2} \right) $
     is fulfilled for each pair $ \vec{x}, \vec{y} \in X$. The definition 
     $ \|\vec{x}\| := \sqrt{ < \vec{x} \: | \vec{x} \: > }$ generates a norm $ \| \cdot \| $
     which makes the pair $ (X, \| \cdot \|) $ to a normed space, and its angle \ $ \angle $ \ has at least all
     properties which has been developed in the previous section. But, of course, the special conditions 
     of an inner product space open some  more possibilities which we try to discover now.   
               
    In Definition \eqref{paare mit reellem winkel} we introduced the set  $ {\cal R}_X $, which includes those pairs 
    $ \vec{x}, \vec{y} \in X $ with a pure real angle. Proposition \eqref{proposition vier}
    ensures the existence of many of such pairs. But now we want more. Roughly spoken we seek for subsets
    $\mathsf{ U } \subset X$, such that  $\mathsf{ U }$ is a real subspace of the complex vector space $ X $ and 
    in addition $\mathsf{ U } \times \mathsf{ U }$ is a subset of   $ {\cal R}_X $, i.e. all products in  
    $\mathsf{ U }$ are real.  
      
    First we take a second look on Proposition \eqref{proposition vier} and its proof.  For two elements \ 
    $  \vec{x}, \vec{y} $ \ of a complex normed space   $ (X, \| \cdot \|) $ we know from  Proposition
    \eqref{proposition vier} that there is at least one $ \varphi \in [ 0, 2 \pi ] $ such that the angle of the pair
    $ ( e^{{\bf i} \cdot \varphi} \cdot \vec{x} \, , \, \vec{y} ) $ is real.   
   \begin{lemma} 
         Let us take  two vectors $ \vec{x}, \vec{y}  $ from an inner product space 
         $ (X, < \cdot \: | \: \cdot >) $ \ with  $ < \vec{x} \: | \vec{y} \:  > \neq 0 $. \
    %     It holds 
    %     $$  \{ ( e^{{\bf i} \cdot \varphi} \cdot \vec{x} \, , \, \vec{y} ) \ | \  \varphi \in [ 0, 2 \pi ] \}  
    %                                    \ \cap \  {\cal R}_X^{\bf \bullet}  \ \neq \ \emptyset \ .         $$  
         We have that there exists one  number \ $ 0 \leq  \varphi < 2 \cdot \pi $ \ such that the set 
         $ Oval(\vec{x},\vec{y}) $ \ has exactly two real angles \\
     $$ \angle\left( e^{{\bf i} \cdot \varphi} \cdot \vec{x} \, , \, \vec{y}\right) \ = \ \frac{\pi}{2} + a , \ \   
        \angle\left( e^{{\bf i} \cdot (\varphi + \pi)} \cdot \vec{x}, \, \vec{y}\right) \ = \ \frac{\pi}{2} - a , \ \  
        \text{with a suitable  number} \ 0 < a \leq  \frac{\pi}{2} \ .  $$    
    \end{lemma} 
    \begin{proof}
        For an inner product $ < \cdot \: | \: \cdot > $ it holds \ 
        $ < e^{{\bf i} \cdot \varphi} \cdot \vec{x} \, | \, \vec{y} > \ = \
        e^{{\bf i} \cdot \varphi} \cdot   < \vec{x} \, | \, \vec{y} > $, hence the cosines   
        $ \cos(Oval(\vec{x},\vec{y})) \subset {\cal B} $ \ build an Euclidean circle with radius \ 
        $ |< \vec{x} \, | \, \vec{y} >| / (\|\vec{x}\| \cdot \|\vec{y}\|) $. \
        We map this set with the  arccosine function, and by Proposition \eqref{proposition eins} the image 
        $ Oval(\vec{x},\vec{y}) \subset {\cal A} $ \  is symmetrical to \ $ \frac{\pi}{2} $, it crosses the real 
        axis exactly two times. (Note that \ $ Oval(\vec{x},\vec{y}) $ \ is no Euclidean circle.)                         
     \end{proof}      
    Albeit we deal with complex inner product spaces we are interested in real subspaces. In a complex 
    normed space  $ (X, \| \cdot \|) $  let $\mathsf{ U } \neq \emptyset$  be any non-empty subset of $ X $. 
  %  inner product space  $ ( X, < \cdot \: | \cdot >) $  let $\mathsf{ U }$  be any subset of $ X $.   
    We define \ $ {\cal L}(\mathbb R)(\mathsf{ U })$ \ as the set of all finite real linear 
    combinations of elements  from $ \mathsf{ U } $, while   \ $ {\cal L}(\mathbb C)(\mathsf{ U })$  is the set of 
    complex linear combinations. In formulas we define
   \begin{align*}
       {\cal L}(\mathbb R)(\mathsf{ U }) \ & := \ 
       \left\{ \sum_{i=1}^{n} r_i \cdot \vec{x}_i \ | \ n \in \mathbb N, \ r_1, r_2, \ldots , r_n \in \mathbb R, \
       \vec{x}_1, \vec{x}_2 , \ldots , \vec{x}_n \in  \mathsf{ U } \right\} , \ \text{while} \\
       {\cal L}(\mathbb C)(\mathsf{ U }) \ & := \ 
       \left\{ \sum_{i=1}^{n} z_i \cdot \vec{x}_i \ | \  n \in \mathbb N, \ z_1, z_2, \ldots , z_n \in \mathbb C, \
       \vec{x}_1, \vec{x}_2 , \ldots , \vec{x}_n \in  \mathsf{ U } \right\} .      
    \end{align*}  
     This definitions mean that $ {\cal L}(\mathbb C)(\mathsf{ U }) $ is a  $\mathbb C$-linear subspace of the               complex vector space $ X $, while  $ {\cal L}(\mathbb R)(\mathsf{ U }) $ only is a real linear subspace of 
     $ X $, which is a real vector space, too.            
                        
     For both spaces we regard the closure in $ X $. Let \ $ \overline{{\cal L}(\mathbb R)(\mathsf{ U })}$ \
     and $ \overline{{\cal L}(\mathbb C)(\mathsf{ U })}$ \ are the closures of \ $ {\cal L}(\mathbb R)(\mathsf{ U }) $ 
     and $ {\cal L}(\mathbb C)(\mathsf{ U }) $, respectively. \ 
     Of course, for a finite set $\mathsf{ U } \subset X$
     it  holds  ${\cal L}(\mathbb R)(\mathsf{ U }) =  \overline{{\cal L}(\mathbb R)(\mathsf{ U })}$ and
     ${\cal L}(\mathbb C)(\mathsf{ U }) =  \overline{{\cal L}(\mathbb C)(\mathsf{ U })}$. 
   %  But the situation changes for infinite sets. More precisely, 
     If we assume an infinite set   $\mathsf{ U }$, an element $ \vec{y} \in X$ belongs to
     $ \overline{{\cal L}(\mathbb R)(\mathsf{ U })}$ if and only if there is a countable set 
     $ \{ \vec{x}_1, \vec{x}_2 , \vec{x}_3, \, \ldots  \} \subset  \mathsf{ U } $
     and there are real numbers $ r_1, r_2, r_3, r_4, \, \ldots $ \ such that 
 \begin{align}    \label{abschluss von lru}
     \lim_{k \rightarrow  \infty}  \left\| \vec{y} - \sum_{i=1}^{k} r_i \cdot \vec{x}_i \right\| = 0 \ . \quad  \ 
     \text{We can write} \ \ \vec{y} \ = \  \sum_{i=1}^{\infty} r_i \cdot \vec{x}_i \ .
 \end{align}      
     The set $ \overline{{\cal L}(\mathbb C)(\mathsf{ U })} $ is constructed similarly, but we can use 
     complex numbers $ z_1, z_2, z_3, \ldots $. \ Again we get two subspaces of $ X $,    
     $ \overline{{\cal L}(\mathbb R)(\mathsf{ U })}$ is a real subspace, while   
     $ \overline{{\cal L}(\mathbb C)(\mathsf{ U })}$ is a complex subspace.  We have inclusions
     $ {\cal L}(\mathbb R)(\mathsf{ U }) \subset \overline{{\cal L}(\mathbb R)(\mathsf{ U })}$, and    
     $ {\cal L}(\mathbb C)(\mathsf{ U }) \subset \overline{{\cal L}(\mathbb C)(\mathsf{ U })}$, respectively,
     and generally both inclusions are proper.  Note  \
     $ \overline{{\cal L}(\mathbb R)(\mathsf{ U })} \subset \overline{{\cal L}(\mathbb C)(\mathsf{ U })}$. 
                                                                               
     In a complex inner product space  $ ( X, < \cdot \: | \: \cdot >) $ we can use the well-known theory of
     orthogonal systems. Informations about this topic can be found in \cite{Rudin} or \cite{Werner}, or any other 
     book about functional analysis. 
    \begin{definition}  \label{definition orthonormal system}
          Let  $ (X, < \cdot \: | \: \cdot >) $ be a complex Hilbert space, i.e. a complex  inner product space
          which is  complete.   A subset \ $ \emptyset \neq \mathsf{ T } \subset X $ is called an orthonormal system
          if and only if  for each pair of distinct elements  $ \vec{x}, \vec{y} \in   \mathsf{ T } \
          ( \text{i.e.} \ \vec{x} \neq \vec{y}) $  it holds 
          $ < \vec{x} \: | \vec{y} \:  > = 0 $,  and all $ \vec{x} \in  \mathsf{ T } $ are unit vectors, i.e. \
          $ \|\vec{x}\| \, = \, 1 \, = \; < \vec{x} \: | \vec{x} \:  > $.  
                            
          An  orthonormal system  $ \mathsf{ T } $ is called an orthonormal basis if and only if  
          $ \mathsf{ T } $ is maximal. This means that if we have  a second  orthonormal system  $ \mathsf{ V } $ 
          with $ \mathsf{ T } \subset \mathsf{ V } $ it has to be $ \mathsf{ T } = \mathsf{ V } $. 
    \end{definition} 
     Note that an orthonormal basis generally is not a vector space basis.  \\ 
     Each unit vector $ \vec{x} $ provides an orthonormal system   $ \{ \vec{x} \} $.   
 %    Proposition \eqref{proposition orthogonal system} 
     It is well known that there is an orthonormal basis  $ \mathsf{ T } $ with  
     $ \{\vec{x}\}  \subset  \mathsf{ T } \subset X $. This shows that there are orthonormal bases in 
     all Hilbert spaces $ X \neq  \{ \vec{0} \} $.   Further note that an orthonormal system 
     $\mathsf{ T } \subset X $ \ is an  orthonormal basis in $ \overline{{\cal L}(\mathbb C)(\mathsf{ T })} $.    \\                                                                                                                                The next proposition describes real subspaces  which have only real inner products. 
     \begin{proposition}  \label{proposition reller winkel}
          Let  $ (X, < \cdot \: | \cdot \:  >) $ be a complex Hilbert space.  Let $ \mathsf{ T } \subset X $ be an
          orthonormal system.  The set  $ \overline{{\cal L}(\mathbb R)(\mathsf{ T })}$  is a real subspace of $ X $,
          and we get that its square \ 
          $ \overline{{\cal L}(\mathbb R)(\mathsf{ T })} \times \overline{{\cal L}(\mathbb R)(\mathsf{ T })} $  
          is a subset of  $ {\cal R}_X $, i.e. each pair  
          $ \vec{y}, \vec{z}  \in  \overline{{\cal L}(\mathbb R)(\mathsf{ T })} \, $, 
          $ \vec{y}, \vec{z} \neq  \vec{0} $, has a real  angle, i.e. $ \angle( \vec{y}, \vec{z}) \in \mathbb R $.   
     \end{proposition}  
     \begin{proof}  
          The set  $ {\cal L}(\mathbb R)(\mathsf{ T })$  is a real subspace of $ X $ by construction, with a 
          vector space basis $ \mathsf{ T } $. The real vector space  ${\cal L}(\mathbb R)(\mathsf{ T })$
          has the closure $ \overline{{\cal L}(\mathbb R)(\mathsf{ T })}$ in $ X $.  
          For \   $ \vec{y}, \vec{z}  \in  \overline{{\cal L}(\mathbb R)(\mathsf{ T })}$ and  $ r \in \mathbb R$ it is
          easy to varify \ $ \vec{y} + \vec{z}  \in  \overline{{\cal L}(\mathbb R)(\mathsf{ T })}$ \ 
          and \  $ r \cdot  \vec{y}  \in  \overline{{\cal L}(\mathbb R)(\mathsf{ T })}$.  This shows that \
          $ \overline{{\cal L}(\mathbb R)(\mathsf{ T })}$  is a real subspace of $ X $. 
                                   
          Now we take two vectors     $ \vec{y}, \vec{z}  \in  \overline{{\cal L}(\mathbb R)(\mathsf{ T })}$ \ with \
          $ \vec{y}, \vec{z} \neq \vec{0}$.  We want to show   $ \angle( \vec{y}, \vec{z}) \in \mathbb R $.                       With line \eqref{abschluss von lru}   we have a countable set \
          $ \{ \vec{x}_1, \vec{x}_2, \vec{x}_3, \vec{x}_4, \ldots \} \subset \mathsf{ T } $ \
          and two sequences \ $ r_1, r_2, r_3, \ldots $ \ and  $ s_1, s_2, s_3, \ldots $ \ of real numbers such that \
          $ \vec{y} \ = \  \sum_{i=1}^{\infty} r_i \cdot \vec{x}_i $ \ and \ 
          $ \vec{z} \ = \  \sum_{i=1}^{\infty} s_i \cdot \vec{x}_i $. 
                                                          
          We compute the cosine of the angle of the pair  \ $ (\vec{y}, \vec{z})$.   
     \begin{align}
           \cos(\angle( \vec{y}, \vec{z})) \ & = \ \cos\left( 
           \arccos{ \left(\frac{< \vec{y} \: | \: \vec{z} >}{ \|\vec{y}\| \cdot \|\vec{z}\| } \right)} \right) \ = \
           \frac{1}{\|\vec{y}\| \cdot \|\vec{z}\| } \ \cdot \ < \vec{y} \: | \: \vec{z} >    \\
            \ & = \ \frac{1}{\|\vec{y}\| \cdot \|\vec{z}\| } \ \cdot \ \
            \left( \ < \ \sum_{i=1}^{\infty} r_i \cdot \vec{x}_i  \: | \: \ 
            \sum_{j=1}^{\infty} s_j \cdot \vec{x}_j  > \ \ \right)   \\     
            \ & = \ \frac{1}{\|\vec{y}\| \cdot \|\vec{z}\| } \ \cdot \ \ 
            \lim_{k \rightarrow \infty} \  \left( \ \sum_{i=1}^{k} \ \sum_{j=1}^{k} \  r_i \cdot  s_j \ \cdot 
             \ < \ \vec{x}_i  \: | \: \ \vec{x}_j  > \ \right) \\  
           \ & = \ \frac{1}{\|\vec{y}\| \cdot \|\vec{z}\| } \ \cdot \ \ 
            \lim_{k \rightarrow \infty} \  \left( \ \sum_{i=1}^{k} \ r_i \cdot  s_i \ \cdot 
             \ < \ \vec{x}_i  \: | \: \ \vec{x}_i  > \ \right) \quad (\mathsf{T} \ \text{is orthonormal}) \\
                \ & = \ \frac{1}{\|\vec{y}\| \cdot \|\vec{z}\| } \ \cdot \ \ 
            \lim_{k \rightarrow \infty} \  \left( \ \sum_{i=1}^{k} \ r_i \cdot  s_i \ \ \right) \ \ 
            \ = \ \frac{1}{\|\vec{y}\| \cdot \|\vec{z}\| } \ \cdot \ \ \sum_{i=1}^{\infty} \ r_i \cdot  s_i \ .      
      %      \ & = \ \frac{1}{\|\vec{y}\| \cdot \|\vec{z}\| } \ \cdot \ \ \sum_{\vec{t} \in \mathsf{ T }}
      %        \left( \ < \  \sum_{i=1}^{\infty} r_i \cdot \vec{x}_i  \: | \vec{t} \: > \, \cdot \, 
      %        < \vec{t} \: | \  \sum_{j=1}^{\infty} s_j \cdot \vec{x}_j  \: > \ \right)   \\     
      %      \ & = \ \frac{1}{\|\vec{y}\| \cdot \|\vec{z}\| } \ \cdot \ \ \sum_{\vec{t} \in \mathsf{ T }}
      %        \left( \ \sum_{i=1}^{\infty} r_i \cdot \ \left[ \sum_{j=1}^{\infty} s_j \ \cdot \
      %       < \ \vec{x}_i  \: | \vec{t} \: > \, \cdot \, < \vec{t} \: | \ \vec{x}_j  \: > \ \right] \ \right) \ . 
      \end{align}    
     %    Note that \ $ \vec{t} \in \mathsf{ T }$ \ and \   
     %    $ \{ \vec{x}_1, \vec{x}_2, \vec{x}_3, \ldots \} \subset \mathsf{ T } $,  i.e. \
     %    $ < \ \vec{x}_i  \: | \vec{t} \: >  = 0 $ \ for \  $ \vec{x}_i  \neq \vec{t} \:$, \ and otherwise \\
     %    \centerline{  $ < \ \vec{x}_i  \: | \vec{t} \: > \ \cdot \ < \vec{t} \: | \ \vec{x}_j  \: > \ =  1 $ \
     %          for \ $ i = j $ \ and \ $ \vec{x}_i  =  \vec{t} =  \vec{x}_j $ .  } \\
         We get \ $ \cos(\angle( \vec{y}, \vec{z}))   =   
   %        \frac{1}{\|\vec{y}\| \cdot \|\vec{z}\| } \ \cdot \ < \vec{y} \: | \: \vec{z} >  \  = \  
         \frac{1}{\|\vec{y}\| \cdot \|\vec{z}\| } \ \cdot \ \sum_{i=1}^{\infty} r_i \cdot  s_i $,
         and we confirm  that indeed the cosine of $ \angle( \vec{y}, \vec{z})  $ is real, this means that the angle 
         $ \angle( \vec{y}, \vec{z}) $ is real,too. Therefore we have proven that  
         $ \overline{{\cal L}(\mathbb R)(\mathsf{ T })} \times \overline{{\cal L}(\mathbb R)(\mathsf{ T })} $  
         is a subset of  $ {\cal R}_X $, and the proof of Proposition \eqref{proposition reller winkel} is finished.  
     \end{proof} 
   The above proposition proves the existence of  many real angles in each complex inner product space. The 
          following statement is in the opposite direction. It shows that the subsets  of 
          $ (X, < \cdot \: | \cdot \: >) $ \ with real angles can not be `arbitrary large'. 
          We formulate the precise statement.
   %%%%%%%%%%%%%%%%%%%%%%%%%%%%%%%%%%%%%%%%%%%%%%%%%%%%%%%%%%%%%%%%%%%%%%%%%%%%%%%%%%%%%%%%%%%%%%%%%%%%       
 %  \newpage       
   \begin{proposition} \label{endlich dimensionale vektorraeume} 
            Let  $ (X, < \cdot \: | \cdot \: >) $ \ be a complex inner product space. Let $ n \in \mathbb N$
            be a natural number.  The following two properties are equivalent. 
       \begin{itemize}
          \item  (1) \ \ There is an orthonormal system 
                 $  {\mathsf T} := \{ \vec{x}_1, \vec{x}_2, \vec{x}_3,  \ldots , \vec{x}_n \} \subset X $ 
                     of \ $ n $ vectors. 
          \item  (2) \ \ There exists a set \
                  $ {\mathsf B} := \{ \vec{v}_1, \vec{v}_2, \vec{v}_3,  \ldots , \vec{v}_n \} \subset X $ \
                  of \ $ n $ elements  such that \ $ {\mathsf B} $ \ is a $ \mathbb R$-linear 
                  independent set and \ $ {\mathsf B} \times {\mathsf B} \subset {\cal R}_X $,
                  i.e. all angles in \ $ {\mathsf B} $ \   are real.    
       \end{itemize}      
   \end{proposition}
       Before we make the easy proof we formulate a corollary. 
       \begin{corollary}   Let \ $ n \in \mathbb N $. \ 
                 We assume any $ n $-dimensional complex inner product space 
                 $(X, < \cdot \: | \cdot \: >) . \ ( \text{Hence the real dimension of the real vector space} \ X \
                 \text{is} \ \ 2 \cdot n ) $.     
                  
                 It follows that it does not be possible to find a set 
                 $ {\mathsf B} := \{ \vec{v}_1, \vec{v}_2, \ldots , \vec{v}_n, \vec{v}_{n+1} \} \subset X $ of
                 $ n+1 $ elements such that both $ {\mathsf B} $ is $\mathbb R$-linear independent, 
       %          $ {\cal L}(\mathbb R)(\mathsf{ B })$ is an $ (n+1)$-dimensional real  subspace of $ X $ 
                 and all angles in $ \mathsf{ B }$ are real. 
       \end{corollary}
       \begin{proof}   
           There are at most  $n$ $\mathbb C$-linear independent vectors in $ X $.  
       \end{proof}
           Now we prove Proposition \eqref{endlich dimensionale vektorraeume}, which is easy.  
       \begin{proof}
             From (1) trivially it follows (2), because for the orthonormal system 
             $ {\mathsf T} = \{ \vec{x}_1, \vec{x}_2, \ldots , \vec{x}_n \} $ it holds that it is 
             $\mathbb C $-linear independent, hence $\mathbb R$-linear independent, and the inner product of
             two elements out of   $ {\mathsf T} $ is either $ 0 $ or $ 1 $. 
                                          
             Also, for complex inner product spaces of infinite dimension the conclusion from (2) to (1) is trivial,
             since in such a space the method of Gram-Schmidt yields  an orthonormal system with the 
             cardinality of $ \mathbb N $.  % See the note after Proposition \eqref{proposition orthogonal system}.  
                       
             We assume a complex inner product space  $ (X, < \cdot \: | \cdot \: >) $ \ with finite complex
             dimension and a set $ {\mathsf B} \subset X $ with the properties of (2), i.e. all inner products in 
             $ {\cal L}(\mathbb R)(\mathsf{ B }) $ are real. \
             We use the very well-known method of Gram-Schmidt to generate a set 
             $ \widehat{{\mathsf B}} := \{ \widehat{{v}_1}, \widehat{v_2}, \widehat{v_3}, 
             \ldots , \widehat{v_n} \} $ \  of $ n $ vectors. 
       %      Since all inner products in $ {\cal L}(\mathbb R)(\mathsf{ B }) $ are real, we remain in this
       %      real subspace. 
             This method yields a set $ \widehat{\mathsf B} $   which spans the same $n$-dimensional real subspace as
             $ {\mathsf B} $, i.e. $ {\cal L}(\mathbb R)(\mathsf{ B }) = {\cal L}(\mathbb R)(\widehat{\mathsf B}) $. \               By construction, the set $\widehat{\mathsf B}$  consists of $ n $ $\mathbb R $-linear independent
             vectors.  Because their inner product $ < \widehat{{v}_i} \: | \widehat{{v}_j} \: > $ is either $ 0 $ 
             or $ 1 $, the set $ \widehat{{\mathsf B}} $ \ even is an orthonormal system as defined in Definition 
             \eqref{definition orthonormal system}. This was required in (1), and 
             Proposition \eqref{endlich dimensionale vektorraeume} is proven.               
       \end{proof}   
       At last we demonstrate that real angles are very useful to do classical Euclidean geometry.  We are still 
       investigating complex inner  product spaces  $(X, < \cdot \: | \cdot \: >) $,  
       and we consider the desirable equation
    \begin{align}  \label{gleichung fuer summe von winkeln}  
          \angle(\vec{x}, \vec{y}) \ = \ \angle(\vec{x}, \vec{x}+\vec{y}) + \angle(\vec{x}+\vec{y}, \vec{y}) \ .  
    \end{align}
           Note that inner product spaces have the property $\overline{(4)}$, the linearity. 
           By using the known formula \\
   \centerline {$ \arccos(r) + \arccos(s) = \arccos\left( r \cdot s - \sqrt{1-r^{2}} \cdot \sqrt{1-s^{2}} \right)$} \\
           for real numbers \ $ -1 \leq r, s \leq +1 $ we can show that
           equation \eqref{gleichung fuer summe von winkeln}
           is fulfilled for a real angle  $ \angle(\vec{x}, \vec{y}) $, i.e.  
           $  < \vec{x} \: | \vec{y} \: > \ \in \, \mathbb R $. \  To demonstrate that a real angle is also 
           necessary, it is sufficient to consider the special case of unit vectors  $ \vec{x}, \vec{y} $. 
           The use of unit vectors simplifies the proof.  
    \begin{proposition} \label{proposition fuer summe von winkeln}
           In a complex inner product space $(X, < \cdot \: | \cdot \: >) $ let $ \vec{x}, \vec{y} $ be two 
           unit vectors. Then it holds equation  \eqref{gleichung fuer summe von winkeln} if and only if \
           $ (\vec{x}, \vec{y}) \in {\cal R}_X $, i.e. their angle $ \angle(\vec{x}, \vec{y}) $  is real.      
    \end{proposition}    
           For the proof first we show a lemma
       \begin{lemma}   \label{gleichung  arckosinus}
           For all numbers $ z \in \mathbb C $ we have the identity \ 
           $  2 \cdot\arccos(z) = \arccos(2 \cdot z^{2} - 1) $. 
       \end{lemma}
       \begin{proof}
           After Definition \eqref{die dritte definition} we wrote \ 
           $ \cos\left( z \right)  =  \frac{1}{2} \cdot \left[  e^{ {\bf i} \cdot z} + e^{ -{\bf i} \cdot z} \right] $,
           and then we can prove easily \ $ \cos\left( 2 \cdot w \right)  =  2 \cdot \cos^{2}(w) - 1 $,  
           for \  $ w \in \mathbb C $. We set \ $ w := \arccos(z) $, and another application 
           of the  arccosine function gives  the desired equation of Lemma \eqref{gleichung  arckosinus}.    
       \end{proof}  
       \begin{proof}
            To prove Proposition \eqref{proposition fuer summe von winkeln}  we assume a complex number \ 
            $ r + {\bf i} \cdot s := \cos(\angle(\vec{x}, \vec{y}))$ \, with two unit vectors  
            $ \vec{x}, \vec{y}$, i.e.  $ \| \vec{x} \| = 1 = \| \vec{y} \| $. 
            By Definition  \eqref{die zweite definition}  of the angle $ \angle(\vec{x}, \vec{y})$ \
            this means   $  r + {\bf i} \cdot s =   <\vec{x} \: | \: \vec{y} > $. \ We consider the right hand side
            of equation \eqref{gleichung fuer summe von winkeln}, and we calculate 
  %%%%%%%%%%%%%%%%%%%%%%%%%%%%%%%%%%%%%%%%%%%%%%%%%%%%%%%%%%%%%%%%%%%%%%%%%%%%%%%%%%%%%%%%%%%%%%%%%%%%%%%%%%%          
    %  \newpage      
        \begin{align*}
              &  \angle(\vec{x}, \vec{x}+\vec{y}) \ +  \ \angle(\vec{x}+\vec{y}, \vec{y}) \  = \ 
                \arccos \left( {\frac{ <\vec{x} \: | \: \vec{x} + \vec{y} > }
                          {\|\vec{x}\| \cdot \|\vec{x} + \vec{y} \| }} \right)  \ + \
                \arccos \left( {\frac{ <\vec{x} + \vec{y}\: | \: \vec{y}  > }
                                                   { \|\vec{x} + \vec{y} \| \cdot \|\vec{y}\| }} \right)  \\
              &   = \ \arccos \left( {\frac{ <\vec{x} \: | \: \vec{x} > +  <\vec{x} \: | \: \vec{y} > }
                      {\| \vec{x} + \vec{y} \| }} \right)  \ + \   \arccos \left( {\frac{ <\vec{x} \: | \: \vec{y} > +
                        <\vec{y} \: | \: \vec{y} > }{ \|\vec{x} + \vec{y} \| }} \right)  \\    
              & = \arccos  \left( 2 \cdot  \left[ {\frac{ 1 +  <\vec{x} \: | \: \vec{y} > }
                 {\| \vec{x} + \vec{y} \| }} \right]^{2}  - 1  \right) 
                                                \quad (\text{by Lemma  \eqref{gleichung  arckosinus}})   \\
              & = \arccos  \left( 2 \cdot  \left[ \frac{ 1 +  2 \cdot ( r + {\bf i} \cdot s ) + 
                             ( r + {\bf i} \cdot s  )^{2} }
                 {  <\vec{x}  + \vec{y} \: | \: \vec{x} + \vec{y} > } \right]  - 1  \right)    \\   
              & = \arccos  \left( 2 \cdot  \left[ \frac{ 1 +  2 \cdot r + r^{2} - s^{2} + 
                                {\bf i} \cdot ( 2 \cdot s +  2 \cdot r  \cdot s )  }
                 { 1 + 2 \cdot r + 1 } \right]  - 1  \right) \quad
                                    (\text{note} \ <\vec{y} | \vec{x}> = \overline{<\vec{x} | \vec{y}>} )   \\   
   %        & ???????????????????????????????????????????????????????????????   \\                         
              & = \arccos  \left( \frac{ r + r^{2} - s^{2} + {\bf i} \cdot ( 2 \cdot s +  2 \cdot r  \cdot s ) }
                                                                                { 1 + r }   \right)      
              \ = \arccos  \left( r - \frac{s^{2}}{1+r}  + {\bf i} \cdot  2 \cdot s  \right)   \ . 
          \end{align*}   
          Obviously, the last term is equal \ $ \arccos(r + {\bf i} \cdot s) = \angle(\vec{x}, \vec{y})$ \   
          only in the case of \ $ s = 0 $, i.e. if and only if the angle $  \angle(\vec{x}, \vec{y})$ is real.
          The proof of Proposition \eqref{proposition fuer summe von winkeln} is finished.        
      \end{proof}   
       With the properties of Theorem \eqref{allererstes theorem} and   %%  as a consequence 
       equation \eqref{gleichung fuer summe von winkeln}  we can do ordinary Euclidean geometry even in 
       complex Hibert spaces in those parts where real angles occur.  We make two examples. \
       As a first example we prove that the sum of inner angles in a triangle is $ \pi $. 
  \begin{proposition}   
        Let us assume a real angle   $  \angle(\vec{x}, \vec{y})$  in a complex Hilbert space. We get 
  $$ \angle(\vec{x}, \vec{y}) + \angle(-\vec{x}, \vec{y}-\vec{x}) +  \angle(-\vec{y}, \vec{x}-\vec{y}) \ = \ \pi . $$     \end{proposition}
  \begin{proof}  
        We use equation \eqref{gleichung fuer summe von winkeln}, and Theorem \eqref{allererstes theorem} 
        ({\tt An} 4),({\tt An} 6),({\tt An} 7). If we regard $\overline{(4)}$, the linearity, we see that 
        all angles in the following equation are real. We compute 
        \begin{align*}
          &    \  \angle(\vec{x}, \vec{y}) \ + \ [\angle(-\vec{x}, \vec{y}-\vec{x})] \ +  \ 
               [\angle(-\vec{y}, \vec{x}-\vec{y})]  \\
               \ = \ & \ \angle(\vec{x}, \vec{y}) \ + \
               \left[ \angle(-\vec{x}, \vec{y}) - \angle(-\vec{x}+\vec{y}, \vec{y}) \right] \ +  \ 
               \left[ \angle(-\vec{y}, \vec{x}) - \angle(-\vec{y}+\vec{x}, \vec{x}) \right]  \\    
               \ = \ & \ \ \pi \quad + \quad 0 \ \ = \ \ \pi \ .  
        \end{align*}
  \end{proof}   
       As a second example we consider the \, `Law of Cosines'. 
    \begin{proposition}               
       Let  $\vec{x}, \vec{y} \neq \vec{0} $ \ be two vectors  in a complex Hilbert space.
       It holds the `Law of Cosines'  
       \begin{align}  \label{kosinussatz}
                \| \vec{x} - \vec{y} \|^{2}  \ = \  \|\vec{x}\|^{2} +  \|\vec{y}\|^{2} 
                - 2 \cdot \|\vec{x}\| \cdot \|\vec{y}\| \cdot \cos (\angle(\vec{x}, \vec{y})) 
       \end{align}                    
       if and only if the angle $ \angle(\vec{x}, \vec{y}) $ is real, i.e. $(\vec{x}, \vec{y}) \in {\cal R}_X$,
       or in other words $ <\vec{x} \: | \: \vec{y}  > \, \in  \mathbb R$. 
    \end{proposition}   
     %  With the  real angle  $ \angle(\vec{x}, \vec{y}) $
    \begin{proof}  
             If the angle  $ \angle(\vec{x}, \vec{y}) $ is real, the proof of the law of cosines is straightforward. 
             In the case of a proper complex angle $ \angle(\vec{x}, \vec{y}) $ the right hand side of 
             equation \eqref{kosinussatz} is complex, while the left hand side is real. 
    \end{proof}     
    Note that the above theorems can be adapted to the complex situation, e.g. the `law of cosines' can be described by
    \ \ $ \| \vec{x} - \vec{y} \|^{2} \ = \ \|\vec{x}\|^{2} + \|\vec{y}\|^{2} - \|\vec{x}\| \cdot \|\vec{y}\| \cdot
    \left[ \cos (\angle(\vec{x}, \vec{y})) + \cos (\angle(\vec{y}, \vec{x}))\right] $, or alternatively \ \
    $ < \vec{x} \: | \: \vec{y} > \ = \ \cos(\angle (\vec{x}, \vec{y})) \cdot  \|\vec{x}\| \cdot \|\vec{y}\| $ .
                                                                                
    At the end we should mention that in the last decades other concepts of generalized angles in
    complex inner product spaces have been considered. Note that for the following concepts of angles 
    $($except $ \angle_3 $, see Theorem \eqref{conjecture csb}$)$ one can use all complex normed spaces, provided
    with the product of Definition \eqref{die erste definition}. 
                                        
    There were attempts to create pure real angles by the definitions \ $ \angle_1 , \ \angle_2 $, \ and $ \angle_3 $.
      \begin{equation*}   
	        \angle_1 (\vec{x}, \vec{y}) \ := \ \text{the real part of} \ (\angle (\vec{x}, \vec{y})) \ = \  
          \text{the real part of} \ 
          \left( \arccos{ \left( \frac{< \vec{x} \: | \: \vec{y} > }{\|\vec{x}\| \cdot \|\vec{y}\|} \right)} \right) , 
      \end{equation*} 
       or alternatively 
      \begin{equation*}   
	        \angle_2 (\vec{x}, \vec{y}) \ := \ \text{the arccosine of the real part of} \ 
	            \left( \frac{< \vec{x} \: | \: \vec{y} > }{\|\vec{x}\| \cdot \|\vec{y}\|} \right) . 
       \end{equation*}   
     The angle \ $ \angle_3 $ \ is defined by 
       \begin{equation*}   
	        \angle_3 (\vec{x}, \vec{y}) \ := \ \arccos(\varrho) \ , \ \ \text{for}  \ \
          \frac{< \vec{x} \: | \: \vec{y} > }{\|\vec{x}\| \cdot \|\vec{y}\|} \ = \ 
          \varrho \cdot e^{{\bf i} \cdot \varphi} \ \in \mathbb C \ .
       \end{equation*} 
       For more information and references see a paper \cite{Scharnhorst} by Scharnhorst. 
                                           
       An interesting position is held by Froda in \cite{Froda}.   For the complex number \  
           $ \frac{< \vec{x} \: | \: \vec{y} > }{\|\vec{x}\| \cdot \|\vec{y}\|} 
           \ = \ r + {\bf i} \cdot s $ \ \ he defined the complex angle 
        \begin{equation*}  
             \angle_4 (\vec{x}, \vec{y}) \ := \ \arccos{ (r) } + {\bf i} \cdot \arcsin{ (s) } \ . 
        \end{equation*} 
        Note that in the case of a pure real non-negative product \ $ < \vec{x} \: | \: \vec{y} > $ \ all four angles
        coincide with our angle  \ $ \angle (\vec{x}, \vec{y})$.  \ \  
        More investigations about these constructions are necessary.  
       { $ $ }   \\  \\
  %%%%%%%%%%%%%%%%%%%%%%%%%%%%%%%%%%%%%%%%%%%%%%%%%%%%%%%%%%%%%%%%%%%%%%%%%%%%%%%%%%%%%%%%%%%%%%%%%%%%%%%%%%%%    
  %  \newpage   
		  {\bf Acknowledgements: } We like to thank  Jeremy Goodsell for reading and correcting parts of the paper,
		  Eberhard  Oeljeklaus for hints and suggestions, and Wilfried Henning for some technical aid.  
  %   who supported us by interested discussions and important advices, and who suggested many improvements.    
  %   Further we thank Berkan G\"urge\c c  for a lot of technical aid. 
  %       \\    
 %%%%%%%%%%%%%%%%%%%%%%%%%%%%%%%%%%%%%%%%%%%%%%%%%%%%%%%%%%%%%%%%%%%%%%%%%%%%%%%%%%%%%%%%%%%%%%%%%%%%%%%%%%%   
  %   \newpage
     \bibliographystyle{mn} 
    
 %%%%%%%%%%%%%%%%%%%%%%%%%%%%%%%%%%%%%%%%%%%%%%%%%%%%%%%%%%%%%%%%%%%%%%%%%%%%%%%%%%%%%%%%%%%%%%%%%  
  \end{document}